\documentclass{amsart}
\usepackage{amsmath,amsthm}
\usepackage{amsfonts,amssymb}
\usepackage{accents}
\usepackage{enumerate}
\usepackage{accents,color}
\usepackage{graphicx}

\usepackage{hyperref}

\newtheorem{thm}{Theorem}[section]
\newtheorem{cor}[thm]{Corollary}
\newtheorem{lem}[thm]{Lemma}
\newtheorem{prop}[thm]{Proposition}

\newtheorem{defn}[thm]{Definition}

\theoremstyle{remark}
\newtheorem{rem}{Remark}[section]

 \def\la{{\langle}}
 \def\ra{{\rangle}}

 \def\d{\mathrm{d}}
 
 \def\sph{{\mathbb{S}^{d-1}}}

 \def\a{{\alpha}}
 \def\b{{\beta}}
 \def\g{{\gamma}}
 \def\k{{\kappa}}
 \def\t{{\theta}}
 \def\l{{\lambda}}
 \def\o{{\omega}}
 \def\s{\sigma}
 \def\la{{\langle}}
 \def\ra{{\rangle}}

 \def\bb{{\mathbf b}}
 
 \def\eb{{\mathbf e}}

 \def\kb{{\mathbf k}}

 \def\ub{{\mathbf u}}
 \def\vb{{\mathbf v}}

 \def\Eb{{\mathbf E}}

 \def\Kb{{\mathbf K}}
 \def\Lb{{\mathbf L}}
 \def\Pb{{\mathbf P}}

 \def\CH{{\mathcal H}}

 \def\CV{{\mathcal V}}

 \def\BB{{\mathbb B}}

 \def\NN{{\mathbb N}}

 \def\RR{{\mathbb R}}
 \def\SS{{\mathbb S}}
 
 \def\VV{{\mathbb V}}

      \def\proj{\operatorname{proj}}

\def\lla{\langle{\kern-2.5pt}\langle}      
\def\rra{\rangle{\kern-2.5pt}\rangle}

\def\la{{\langle}}
\def\ra{{\rangle}}

\def\d{\mathrm{d}}

\def\sph{{\mathbb{S}^{d-1}}}

\def\sE{{\mathsf E}}
\def\sF{{\mathsf F}}
\def\sG{{\mathsf G}}

\def\sK{{\mathsf K}}
\def\sL{{\mathsf L}}

\def\sP{{\mathsf P}}

\def\sS{{\mathsf S}}

\def\sb{{\mathsf b}}
\def\sc{{\mathsf c}}
\def\sd{{\mathsf d}}
\def\sm{{\mathsf m}}
\def\sw{{\mathsf w}}

\def\fD{{\mathfrak D}}

\def\a{{\alpha}}
\def\b{{\beta}}
\def\g{{\gamma}}
\def\k{{\kappa}}
\def\t{{\theta}}
\def\l{{\lambda}}
\def\o{{\omega}}
\def\s{\sigma}
\def\la{{\langle}}
\def\ra{{\rangle}}

\def\bb{{\mathbf b}}

\def\kb{{\mathbf k}}

\def\ub{{\mathbf u}}
\def\vb{{\mathbf v}}

\def\Eb{{\mathbf E}}

\def\Jb{{\mathbf J}}
\def\Kb{{\mathbf K}}
\def\Lb{{\mathbf L}}

\def\Pb{{\mathbf P}}

\def\CH{{\mathcal H}}

\def\CV{{\mathcal V}}

\def\BB{{\mathbb B}}

\def\NN{{\mathbb N}}

\def\RR{{\mathbb R}}
\def\SS{{\mathbb S}}

\def\VV{{\mathbb V}}

\def\proj{\operatorname{proj}}

\def\lla{\langle{\kern-2.5pt}\langle}      
\def\rra{\rangle{\kern-2.5pt}\rangle}      

\newcommand{\wt}{\widetilde}
\newcommand{\wh}{\widehat}

\def\fD{{\mathfrak D}}

\def\f{\frac}

\def\f{\frac}

\graphicspath{{./}}
\begin{document}

\title{Best approximation by polynomials on the conic domains} 
\author{Yan Ge}
\address{Department of Mathematics, University of Oregon, Eugene, 
OR 97403--1222, USA}
\email{yge7@uoregon.edu} 
\author{Yuan~Xu}
\address{Department of Mathematics, University of Oregon, Eugene, 
OR 97403--1222, USA}
\email{yuan@uoregon.edu} 
\thanks{The second author was partially supported by Simons Foundation Grant \#849676.}
\date{\today}  
.

\begin{abstract}  
A new modulus of smoothness and its equivalent $K$-function are defined on the conic domains in $\RR^d$,
and used to characterize the weighted best approximation by polynomials. Both direct and weak inverse 
theorems of the characterization are established via the modulus of smoothness. 
For the conic surface $\VV_0^{d+1} = \{(x,t): \|x\| = t\le 1\}$, the natural weight function is $t^{-1}(1-t)^\g$, 
which has a singularity at the apex, the rotational part of the modulus of smoothness is defined in terms of 
the difference operator in Euler angles with an increment $h/\sqrt{t}$, akin to the Ditzian-Totik modulus on 
the interval but with $\sqrt{t}$ in the denominator, which captures the singularity at the apex. 
\end{abstract}
\maketitle

\section{Introduction}
\setcounter{equation}{0}

One essential problem in approximation theory is understanding the correlation between the rate of an approximation process 
and the smoothness of the function being approximated. For trigonometric polynomial approximation or algebraic polynomial 
approximation on the real line, the relation is well understood through the direct and inverse theorems that characterize 
the best approximation via either a modulus of smoothness or a $K$-functional (cf. \cite{DL, DT}). For polynomial approximation 
in higher dimensions, a similar characterization has been established fully for a handful of regular domains, most notably 
the unit sphere and the unit ball (cf. \cite{DaiX2, DaiX, LN, Rus}), and characteristically for some general class of domains, 
such as polytops \cite{T1, T2} and connect $C^2$ domain \cite{DP}. 

In this paper, we consider weighted approximation on two conic domains that are standardized as 
\begin{enumerate}[\quad (1)]
\item Conic surface defined by 
$$
	\VV_0^{d+1}: = \left\{(x,t) : \|x\| = t, \, 0 \le t \le 1, \, x\in \RR^d\right\},
$$ 
where $\|x\|$ denotes the usual Euclidean norem. 
\item Solid cone bounded by $\VV_0^{d+1}$ and the hyperplane $t =1$, defined by 
$$
	\VV^{d+1}: = \left\{(x,t): \|x\| \leq  t, \, 0 \le t \le 1, \, x\in \RR^d\right\}.
$$
\end{enumerate}
Both domains have a singularity at the apex, which may affect how the best approximation can be characterized.

While the analysis in the unit sphere and the unit ball has a long history, the study in the conic domains becomes 
feasible only after recent advances in the orthogonal structure of polynomials in these domains. For a class of 
weight functions, it is shown in \cite{X20} that orthogonal polynomials on the conic domains are eigenfunctions of 
a second-order differential operator, called the spectral operator, and their reproducing kernels satisfy an addition formula. 
These two properties share a common root in classical orthogonal polynomials \cite{DX} and are essential for studying 
the Fourier orthogonal series. Based on these two properties, a general framework is developed in \cite{X21} for 
a space of homogeneous type that allows us to define, in particular, a modulus of smoothness and an equivalent 
$K$-functional, the latter defined via the spectral operator. These are used to characterize the weighted best 
approximation by polynomials under the assumption of highly localized kernels, which are typically established via 
the addition formula. The framework generalizes the results in classical domains, such as the unit sphere and 
the unit ball, and it also covers the conic domains. The characterization, however, is in line with the classical one 
on the unit sphere (cf. \cite{BBP, LN, P, Rus, X05} and the references therein), in which the modulus of smoothness 
is defined as a multiplier operator of the Fourier orthogonal series, which is, for weighted approximation, 
a far cry from the traditional modulus of smoothness defined via differences in function evaluations, and does 
not connect to the smoothness of the function as transparently as one would like. 

On the unit sphere with constant weight function, a more intuitive modulus of smoothness is defined in \cite{DaiX2} 
via the differences in Euler angles, which reduces to classical differences on the largest circles on the sphere. This 
modulus of smoothness leads to a characterization of the best approximation by polynomials and its equivalent 
$K$-functional is defined via derivatives in Euler angles, which can be used to give a decomposition of the spectral 
operator. Moreover, working with the polar coordinates so that the unit ball can be regarded as a semi-product of 
the sphere of radial $r$ and the interval on $[0,1]$, a new modulus of smoothness on the unit ball is also derived 
in \cite{DaiX2} and used to characterize the best approximation. 

Since the conic surface $\VV_0^{d+1}$ can also be regarded as a semi-product of spheres of radial $t$ and the 
interval of $t$ in $[0,1]$, it is foreseeable that a new modulus of smoothness, different from the one defined via the 
multiplier operator, can be defined and utilized for characterizing weighted best approximation. This is the starting 
point for our study. The problem remains challenging since the geometry of the cone has its peculiarity with a  
singular point at its apex, which causes technical difficulties that must be overcome. For starters, our new modulus
of smoothness on the conic surface will be a fusion of moduli of smoothness on the sphere and the interval. 
However, the conic surface is not a direct product of the sphere and the interval, so it is not immediately 
clear how the existing results on the sphere and the interval can be utilized. It turns out that the difference 
operators in Euler angles need to have an increment $h/\sqrt{t}$, which is a new phenomenon, so that the modulus 
of smoothness is akin to the Ditzian-Totik modulus on the interval, but with $\sqrt{t}$ in the denominator. This 
phenomenon looks to be novel, and likely what is needed for the domain with this type of singularity. 
Using this modulus of smoothness and its equivalent $K$-functional of the first and the second order, 
our main results establish the direct and inverse theorems for the best approximation by polynomials in 
weighted $L^p$ spaces on the conic domains. 

The paper is organized as follows. In the next section, we recall the background materials and results needed 
for our study. This includes the best approximation on the interval, where we will use the Ditzian-Totik modulus of 
smoothness for the Jacobi weight, which we need even for the constant weight function on the conic domain and 
on spheres and balls, as well as an orthogonal structure on the conic domains. The main results on the conic surface 
will be stated and proved in the third section, and the corresponding material on the solid cone will be given in the 
fourth section.  

Throughout this paper, we denote by $L^p(\Omega, \sw)$ the weighted $L^p$ space with respect to the weight 
function $\sw$ defined on the domain $\Omega$ and its norm is denoted by $\|\cdot\|_{p,\sw}$ for $1 \le p \le \infty$
with the understanding that the space is $C(\Omega)$ with the uniform norm when $p= \infty$. 

For conical domains, in an attempt to distinguish conic surfaces and solid cones, we shall denote the operator 
on the surface in sans serif font, such as $\sP_n$ and $\sL_n$, and the operator on the solid domains in bold font, 
such as $\Pb_n$ and $\Lb_n$. 

Finally, we shall use $c$, $c'$, $c_1$, $c_2$, etc. to denote positive constants that depend on fixed 
parameters, and their values may change from line to line, and we shall write $A \sim B$ if $ c' A \le B \le c A$. 

\section{Preliminary and Lemmas} \label{sec:2}
\setcounter{equation}{0}

Our main study will require results for the best approximation by polynomials on the interval $[0,1]$, with the Jacobi 
weight function, and on the unit sphere. These will be reviewed in the first two subsections. Orthogonal structure 
on the conic surface will be recalled in the third subsection, and the highly localized kernels on the conic surface  
will be reviewed in the fourth subsection, where several new properties needed in this study will be established. 

\subsection{Approximation on the interval $[0,1]$} \label{sec:Approx[0,1]}
Let $\Pi_n$ be the space of polynomials of degree at most $n$. Let $\varpi$ be a weight function on the interval $[0,1]$.
For $f\in L^p(\varpi,[0,1])$, $1 \le p< \infty$, or $f\in C[0,1]$ for $p = \infty$, we consider the quantity of best approximation
by polynomials defined by 
$$
   E_n(f)_{p,\varpi} = \inf_{p \in \Pi_n} \|f - p\|_{p,\varpi}, \qquad 1 \le p \le \infty. 
$$
The characterization of this quantity has been extensively studied; see \cite{DL, DT} and the references therein. It is 
known early on that the difficulty lies in the boundary behavior of approximation. Among various attempts to resolve
the problem, the Ditzian-Totik modulus of smoothness, and its equivalent $K$-functional, is most natural and 
satisfactory. To recall its definition, we need the central difference operator $\Delta_\t$ with $\t> 0$ defined 
by $\Delta_\t f(t)=f\left(t+\f \t2\right)-f \left(t-\f \t2\right)$ and, for $r \in \NN$, the $r$-th central difference 
$\Delta_\t^r$ defined by
$$
   \Delta_{\t}^r f (t) :=  \Delta_{\t}^{r-1} \Delta_\t f (t)= \sum_{j=0}^r (-1)^j \binom{r}{j} f\Big(t+ \left(\frac{r}2 -j\right)\t  \Big).
$$ 
Throughout this paper, we define by $\varphi$ the function
$$
      \varphi(t):= \sqrt{t(1-t)}, \qquad 0 \le t \le 1. 
$$
The Ditzian-Totik modulus of smoothness is defined via the central difference with the increment $\t$ replaced 
by the function $\t \varphi(t)$, which is the feature that makes it stand out. More precisely, for $f\in L^p([0,1])$, we define
\begin{equation} \label{eq:DT-modulus}
  \omega_\varphi^r (f; t)_p = \sup_{0< \t < t} \left\| \Delta^r_{\t\varphi} f \right\|_{L^p[0,1]},
\end{equation} 
where $\Delta^r_{\t\varphi} f (t) = 0$ if $t - r \t \varphi(t)/2 \notin [0,1]$ and $1-t+ r\t \varphi(t)/2 \notin [0,1]$. 
This modulus of smoothness provides a characterization for $E_n(f)_{p}$, which is $E_n(f)_{p, \varpi}$ with $\varpi(t) =1$.
More precisely, it gives the direct theorem \cite[Theorem 7.2.1]{DT}
\begin{equation} \label{eq:direct[0,1]}
   E_n(f)_p \le c_p   \omega_\varphi^r \left(f;  n^{-1}\right), \qquad 1 \le p \le \infty
\end{equation} 
and a matching standard weak inverse theorem \cite[Theorem 7.2.4]{DT}
\begin{equation} \label{eq:inverse[0,1]}
     \omega_\varphi^r \left(f;  n^{-1}\right) \le c n^{-r} \sum_{k=0}^n (k+1)^{r-1} E_k(f)_p.
\end{equation} 

For our study, we shall need the weighted  modulus of smoothness for the Jacobi weight function, 
$$
    \varpi_{\a,\b}(t) = t^\a (1-t)^\b, \qquad \a, \b > -1,
$$
definded on the inverval $[0,1]$. 
Then, the main-part  modulus of smoothness for the Jacobi weight is defined by \cite[(8.1.2)]{DT}
\begin{align} \label{eq:o_phi}
\o^r_\varphi (f; t)_{p,\varpi_{\a,\b}} := \sup_{0<\t \le t} \left\| \Delta_{\t\varphi}^r f\right\|_{L^p(J_{r t},\varpi_{\a,\b}) } 
\end{align} 
where the  norm is  taken over the interval defined by 
$$
J_ {r t}:=[12 r^2 t^2, 1- 12 r^2 t^2].
$$
The direct and inverse theorems were established using the main-part moduli in 
\cite[Theorem 8.2.1]{DT}. What is interesting for our study is the weighted $K$-functional defined by
\begin{align} \label{eq:K_phi}
   K^r_\varphi (f; t)_{p,\varpi_{\a,\b}} := \inf_{g \in C^r[0,1]}
  	\left \{ \|f-g\|_{p,\varpi_{\a,\b}} + t^r \left\|\varphi^r g^{(r)} \right\|_{p,\varpi_{\a,\b}} \right \},
\end{align}
which is equivalent to the main-part modulus of smoothness. More precisely, 
\begin{equation} \label{eq:K-o[0,1]}
   c_1 \o^r_\varphi (f; t)_{p,\varpi_{\a,\b}} \le K^r_\varphi (f; t)_{p,\varpi_{\a,\b}} \le c_2 \o^r_\varphi (f; t)_{p,\varpi_{\a,\b}}
\end{equation}
follows from \cite[Theorem 6.1.1]{DT} since the weight function and \(\varphi\) adhere to the conditions stipulated 
in \cite[(6.1.2) and Section 1.2]{DT}. For the unweighted case of $\a=\b=0$, the equivalence is in 
\cite[Theorem 2.1]{DT}. 

We will also need the Jacobi polynomials $P_n^{(\a,\b)}$, given explicitly by 
$$
  P_n^{(\a,\b)}(t) = \frac{(\a+1)_n}{n!} {}_2F_1 \left (\begin{matrix} -n, n+\a+\b+1 \\
      \a+1 \end{matrix}; \frac{1-t}{2} \right)
$$
via the hypergeometric function, and they are orthogonal with respect to the weight function 
$w_{\a,\b}(t)= (1-t)^\a(1+t)^\b = 2^{\a+\b}\varpi_{\a,\b}(\frac{1+t}{2})$,  
$$
  c_{\a,\b}' \int_{-1}^1 P_n^{(\a,\b)}(t) P_m^{(\a,\b)}(t) w_{\a,\b}(t) \d t = h_n^{(\a,\b)} \delta_{n,m},
$$
where  $c'_{\a,\b} = 2^{-\a-\b-1} c_{\a,\b}$ with 
\begin{align} \label{eq:JacobiNorm}
  c_{\a,\b} := \frac{\Gamma(\a+\b+2)}{\Gamma(\a+1)\Gamma(\b+1)}\quad\hbox{and}\quad  h_n^{(\a,\b)}:= \frac{(\a+1)_n (\b+1)_n(\a+\b+n+1)}{n!(\a+\b+2)_n(\a+\b+2 n+1)}.
\end{align}
For later use, we define 
\begin{equation}\label{eq:Zn} 
  Z_{n}^{(\a,\b)}(t) := \frac{P_n^{(a,\b)}(t) P_n^{(\a,\b)}(1)}{h_n^{(\a,\b)}},  \quad \a,\b > -1,
\end{equation}
where $h_n^{(\a,\b)}$ is the $L^2$ norm of $P_n^{(\a,\b)}$ in $L^2([-1,1],w_{\a,\b})$. 

One of the important tools for our study is the highly localized kernel, defined as a resample of the reproducing kernel. 
Let $\wh a \in C^\infty$ be a cut-off function that satisfies 
\begin{equation} \label{eq:cut-off}
 \wh a(t) \ge 0, \quad t\in \RR_+, \qquad \wh a (t) = 1, \quad 0 \le t\le 1, \qquad \wh a(t) = 0, \quad t\ge 2.
\end{equation}
Then the highly localized kernel $L_n^{(\a,\b)}$ is defined by, for $\a, \b \ge - \f12$,  
\begin{equation} \label{eq:Ln-Jacobi}
  L_n^{(\a,\b)}(s,t) = \sum_{k=0}^{2n} \wh a \left(\frac k  n \right) \frac{P_k^{(a,\b)}(t) P_k^{(\a,\b)}(s)}{h_k^{(\a,\b)}}.
\end{equation}
It is highly localized in the sense that for every $\k > 0$, there is a constant $c_\k$ such that
\begin{equation} \label{eq:kernel_Jacobi}
  \left| L_n^{(\a,\b)}(s,t) \right| \le c_\k \frac{n}{ \sqrt{\varpi_{\a,\b}(n;s)\varpi_{\a,\b}(n;t)} (1+n \sd_{[-1,1]}(s,t))^\k},
\end{equation}
where the function $\varpi_{\a,\b}(n;t)$ is defined by 
$$
    \varpi_{\a,\b}(n;t) = (1-t+n^{-2})^{\a+\f12} (1+t+n^{-2})^{\b+\f12}, \qquad -1<t<1,
$$
and $\sd_{[-1,1]} (s,t) = \arccos(s t + \sqrt{1-s^2} \sqrt{1-t^2})$ is the distance of $[-1,1]$; see \cite[Theorem 2.4]{PX}.
For our study, the following lemma will play an important role. 

\begin{lem} \label{eq:int_Jacobi_ker}
For $\a,\b\ge -\f12$ and $\g,\delta \ge 0$, 
$$
   \int_{-1}^1  |L_n^{(\a,\b)}(s,t)| \varpi_{\a+\g,\b+\delta}(t) \d t \le c \varpi_{\g-\f12,\delta-\f12}(n;s). 
$$
\end{lem}

\begin{proof}
For $\g = \delta =0$, this is the case $p = 1$ of \cite[Proposition 2.5]{PX}. We modify the proof therein for the 
general case. We divided the integral to be estimated as a sum of two integrals to decouple the two endpoints. 
By symmetry, we only need to consider one of them, which we assume to be the integral over $[-\f12,1]$, for
which we can assume $s \in [0,1]$ correspondingly. 
Changing variables $t = \cos \t$, $s = \cos \phi$, we obtain $\sd_{[-1,1]}(s,t) = | \t- \phi |$ and, moreover, by
$\sin \t \sim \t$ for $0 \le \t \le 2 \pi/3$, $\varpi_{\a,\b}(n;t) \sim (\sin^2 \f{\t}{2} + n^{-2})^{\a+\f12}\sim (\t+ n^{-1})^{2\a+1}$.
Hence, by \eqref{eq:kernel_Jacobi}, 
\begin{align*}
 J: = &  \int_{-\f12}^1 |L_n^{(\a,\b)}(s,t)| \varpi_{\a+\g,\b+\delta}(t) \d t  \\
 & \le c \frac{n}{\varpi_{\a,\b}(n;s)} \int_0^{\frac{2\pi}3} \frac{(\sin \t)^{2\a + 2\g+1}}
     {(\sin \t + n^{-1})^{\a+\f12} (1+n|\t-\phi|)^\k} \d \t \\
 & \le c \frac{n}{\varpi_{\a,\b}(n;s)} \int_0^{\frac{2\pi}3} \frac{\t^{2\a + 2\g+1}}{(\t + n^{-1})^{\a+\f12} (1+n|\t-\phi|)^\k} \d \t,
\end{align*}
where the last inequality follows from $\sin \t \sim \t$ for $0 \le \t \le 2 \pi/3$. Changing variable $\t\mapsto n \t=u$, it follows
\begin{align*}
 J \,& \le c \frac{n^{\a+\f32}}{\varpi_{\a,\b}(n;s) n^{2 \a+2\g+2}} 
    \int_0^{\frac{2\pi}3 n}\frac{u^{\a +2\g+\f12}}{  (1+|u -n \phi|)^\k} \d u \\
     & =  c \frac{1}{\varpi_{\a,\b}(n;s) n^{\a+2\g+\f12}} \int_{-n \phi}^{n (\frac{2\pi}{3}-\t)}
           \frac{(v+n \phi)^{\a + 2\g+1}}{ (1+|v|)^\k} \d v.
\end{align*}
Hence, using $(A + B)^q \le  2^q(A^q + B^q)$, we conclude that 
\begin{align*}
 J \,& \le c \frac{1}{\varpi_{\a,\b}(n;s) n^{\a+2\g+\f12}} \left[ \int_{-\infty}^\infty  \frac{1}{ (1+|v|)^{\k-\a-2\g-1}} \d v
   +  (n \phi)^{\a+ 2\g+\f12} \int_{-\infty}^\infty  \frac{1}{ (1+|v|)^\k} \d v \right] \\
    & \le  c \frac{1}{\varpi_{\a,\b}(n;s) } (\phi+n^{-1})^{\a+2\g+\f12} \le c (\sin \phi + n^{-1})^{2\g} \le c\, \varpi_{\g-\f12,\delta-\f12}(n;s).
\end{align*}
This completes the proof.
\end{proof}

\subsection{Approximation on the sphere}
Let $\sph = \{\xi \in \RR^d: \|\xi\| = 1\}$ denote the unit sphere of $\RR^d$. For $f\in L^p(\sph)$, $1\le p <\infty$, 
or $f\in C(\sph)$, $p = \infty$, we consider the best approximation defined by 
$$
   \sE_n(f)_{L^p(\sph)} = \inf_{Y\in \Pi_n(\sph)} \| f - Y\|_p, \qquad 1 \le p \le \infty,
$$
where $\Pi_n(\sph)$ denote the space of polynomials restricted over $\sph$. Spherical harmonics are  
homogeneous harmonic polynomials restricted to the unit sphere. Let $\CH_n^d$ denote the space of 
spherical harmonics of degree $n$. Spherical harmonics of different degrees are orthogonal; more precisely,
if $Y_n \in \CH_n^d$, then 
$$
  \int_\sph Y_n(\xi) Y_m(\xi) \d \s(\xi) = 0, \quad n \ne m,
$$
where $\d \s$ denotes the surface measure. The space $\Pi_n(\sph)$ can be decomposed as a direct sum of 
$\CH_k^d$ for $0 \le k \le n$. It is known that 
\begin{equation*} 
 \dim \CH_n^d =\binom{n+d-1}{n} - \binom{n+d-3}{n-2}.  
\end{equation*}   
A fundamental property states that the spherical harmonics are eigenfunctions of the Laplace-Beltrami 
operator $\Delta_0$ \cite[(1.4.9)]{DaiX}, 
\begin{equation} \label{eq:sph-harmonics}
     \Delta_0 Y = -n(n+d-2) Y, \qquad Y \in \CH_n^d,
\end{equation}
where $\Delta_0$ is the restriction of $\Delta$ on the unit sphere. Another important property is that
they satisfy an addition formula \cite[(1.2.3) and (1.2.7)]{DaiX}: for $\xi \in\sph$ and $\eta \in \sph$, 
\begin{equation} \label{eq:additionF}
   \sum_{\ell =1}^{\dim \CH_n^d} Y_\ell^n (\xi) Y_\ell^n(\eta) = Z_n^{\f{d-2}{2}} (\la \xi,\eta\ra), \qquad 
    Z_n^\l(t) = \frac{n+\l}{\l} C_n^\l(t),
\end{equation}
where $C_n^\l$ is the classical Gegenbauer polynomial, which corresponds to the Jacobi polynomial
with $\a=\b= \l-\f12$ and, in particular, $Z_n^\l(t) = Z_n^{(\l-\f12, \l-\f12)}(t)$ in \eqref{eq:Zn}. 

The left-hand side of \eqref{eq:additionF} is the reproducing kernel of $\CH_n^d$ and the kernel of 
the projection operator $\proj_n: L^2(\sph) \to \CH_n^d$: 
$$
   \proj_n f(\xi)  = \frac{1}{\s_d}\int_{\sph} f(y) \sP_n(\xi,\eta) \d\s_\SS(\eta), \qquad   
       \sP_n (\xi,\eta) =  \sum_{\ell =1}^{\dim \CH_n^d} Y_\ell^n (\xi) Y_\ell^n(\eta),
$$
where $\s_{d}$ is the surface area of $\sph$, $\s_d= 2 \pi^{d/2}/\Gamma(d/2)$. 
The Fourier orthogonal series for $f\in L^2(\sph)$ is given by 
$$
   L^2(\sph) = \bigoplus_{n=0}^\infty \CH_n^d\, : \qquad f = \sum_{n=0}^\infty \proj_n f. 
$$ 

The first modulus of smoothness that characterizes $E_n(f)_p$ on the sphere is defined via the spherical means
(\cite[p. 216]{BBP} and \cite[p. 475]{P}) 
\begin{equation*} 
 \sS_\theta f(x) = \frac{1}{\s_{d} (\sin \theta)^{d-2}} \int_{\langle x, y\rangle = \cos\theta} f(y) \d\s_{x,\theta}(y),
\end{equation*}
where $\d \s_{x,\theta}$ is the Lebesgue measure on  $\{ y\in \SS^{d-1}: \langle x, y\rangle = \cos\theta\}$. 
For $r > 0$, the modulus of smoothness is defined by
\begin{equation}\label{eq:omega*}
  \o_r^*(f; t)_p := \sup_{|\t| \le t} \| (I - S_\t)^{r/2} f\|_p,
\end{equation} 
where $(I-S_\t)^{ r/2}$ is defied as an infinite series when $ r/2$ is not an integer. The $K$-functional equivalent to 
$\o_r^*(f, t)_p$ is defined by 
\begin{equation}\label{eq:K-func*}
  \sK_r^*(f; t)_p := \inf_g \left\{ \|f-g\|_p + t^r \|(- \Delta_0)^{r/2}g\|_p \right\}.
\end{equation}
The characterization of $\sE_n(f)_p$ and the equivalence of $\o_r^*(f; t)_p$ and  $\sK_r^*(f; t)_p$ is established 
in its final form in \cite{Rus}, built upon early works of many authors (see, for example, \cite{BBP, LN, P, Rus} and 
the references therein). The spherical means $\sS_\theta f$ is a multiplier operator of the Fourier series in 
spherical harmonics as it satisfies
\begin{equation}\label{eq:multiplier_S}
   \proj_n \left(\sS_\theta f \right) = \frac{C_n^{\l}(t)}{C_n^\l(1)} \proj_n f, \qquad \l = \frac{d-2}{2}.
\end{equation}
This aspect of $\sS_\t$ is extended and used to define a modulus of smoothness in the weighted setting 
of the sphere with Dunkl's reflection invariant weight function in \cite{X05}. The weighted modulus of smoothness, 
however, can no longer be associated with simple properties of functions. 

More recently, another modulus of smoothness on the unit sphere is defined in \cite{DaiX2} in terms of the
Euler angles. In a nutshell, it is defined through forward differences in the largest circles of $\sph$ intersecting
with $(x_i, x_j)$-plane for $1\le i < j \le d$. More precisely, let $Q_{i,j, \t}$ denote the rotation by the angle $\t$
in the $(x_i,x_j)$ plane of $\RR^d$ so that the rotation angle from the coordinate vector $e_i$ to $e_j$ is assumed 
to be positive. For $r = 1,2,\ldots$, we define the difference operator
$$
  \Delta_{i,j,\t}^r f = \left(I - T_{Q_{i,j,\t}} \right)^r f, \quad 1 \le i \ne j \le d, 
$$
where $T_Q$ dentoes rotation operator $T_Q f(x)= f(Q^{-1}x)$. Since $Q_{i,j,\t}$ is a rotation that acts on the 
$(x_i, x_j)$-plane, $\Delta_{i,j,\t}$ is a forward difference in the $\t_{i,j}$ variable. For example, if $(i,j) = (1,2)$ 
and $(x_1,x_2) = (t \cos \phi, t \sin \phi)$, then 
\begin{equation}\label{eq:Delta_ijt}
  \Delta_{i,j,\t}^r f(x)  = \overrightarrow{\Delta}_\t^r f \big(t \cos (\phi+ \cdot), t \sin (\phi+ \cdot), x_3, \ldots, x_d\big), 
\end{equation}
where $\overrightarrow{\Delta}_\t^r$ is actigng on the variable $(\cdot)$. For $r \in \NN$, the modulus
of smoothness on $\sph$ is defined by \cite{DaiX2}
\begin{equation} \label{eq:o_rSph}
 \omega_r^\SS(f; t)_p := \max_{1\le i < j \le d} \sup_{|\t| \le t} \left \| \Delta_{i,j,\t}^r f\right\|_{L^p(\sph)}, \quad 1 \le p \le \infty.   
\end{equation}
The $K$-functional equivalent to $ \omega_r(f; t)_p$ is defined in terms of angular derivatives 
$$
  D_{i,j} = x_i  \frac{\partial}{\partial x_j} - x_j \frac{\partial}{\partial x_i} = \frac{\partial}{\partial \t_{i,j}},  \qquad 1 \le i,j \le d,
$$
where $\t_{i,j}$ is the angle of polar coordinates in the $(x_i,x_j)$-plane, defined by $(x_i,x_j) = r_{i,j} (\cos \t_{i,j},\sin \t_{i,j})$.
More precisely, 
\begin{equation} \label{eq:K_rSph}
 \sK_r^\SS (f, t)_{p}: = \inf_{g \in C^r(\sph)} \left \{ \|f-g\|_{L^p(\sph)} + t^r  \max_{1\le i < j \le d} \left\|D_{i,j}^r g \right\|_{L^p(\sph)} \right \}.
\end{equation}

Both the direct theorem and the inverse theorem are established in \cite{DaiX2} for this pair of modulus of smoothness 
and $K$-functional, and it is also proved that 
\begin{equation} \label{eq:o_rSph=K_rSph}
  c_1 \sK_r^\SS (f, t)_{p} \le  \omega_r^\SS(f; t)_p  \le c_2 \sK_r^\SS (f, t)_{p}, \qquad 1 \le p \le \infty. 
\end{equation}
The relation between this pair and the  pair \eqref{eq:omega*} and \eqref{eq:K-func*} is discussed in 
\cite{DaiX2}, which uses the decomposition of the Laplace-Beltrami operator \cite[(1.8.3)]{DaiX}
\begin{equation} \label{eq:Delta0=sum}
      \Delta_0 = \sum_{1 \le i< j \le d} D_{i,j}^2.
\end{equation}

\subsection{Orthogonal structure on the conic surface}
Let $\Pi(\VV_0^{d+1})$ denote the space of polynomials restricted on the conic surface $\VV_0^{d+1}$
and, for $n = 0,1,2,\ldots$, let $\Pi_n(\VV_0^{d+1})$ be the subspace of polynomials in $\Pi(\VV_0^{d+1})$
of total degree at most $n$. Since $\VV_0^{d+1}$ is a quadratic surface, it is known that
$$
 \dim \Pi_n (\VV_0^{d+1}) = \binom{n+d}{n}+\binom{n+d-1}{n-1}.
$$
For $\b > - d$ and $\g > -1$, we define the weight function $\sw_{\b,\g}$ by  \footnote{This is the same Jacobi weight 
denoted by $\varpi_{\b,\g}$ in Subsection \ref{sec:Approx[0,1]}. We adopt the sans serif font when it is used as 
a weight function on the conic domain.}
$$
   \sw_{\b,\g}(t) = t^\b (1-t)^\g, \qquad 0 \le t \le 1.
$$
Orthogonal polynomials with respect to $ \sw_{\b,\g}$ on $\VV_0^{d+1}$ are studied in \cite{X20}. Let 
$$
\la f, g\ra_{\sw_{\b,\g}} =\sb_{\b,\g} \int_{\VV_0^{d+1}} f(x,t) g(x,t) \sw_{\b,\g} \d \sm(x,t),
$$ 
where $\d \sm$ denotes the Lebesgue measure on the conic surface and $\sb_{\b,\g}$ is the normalizatin
constant so that $\la 1,1\ra_{\sw_{\b,\g}}=1$. This is a well-defined inner product 
on $\Pi(\VV_0^{d+1})$. Let $\CV_n(\VV_0^{d+1},\sw_{\b,\g})$ be the space of orthogonal polynomials of degree 
$n$ under $\la \cdot, \cdot\ra_{\sw_{\b,\g}}$. Then $\dim \CV_0(\VV_0^{d+1},\sw_{\b,\g}) =1$ and 
$$
   \dim \CV_n(\VV_0^{d+1},\sw_{\b,\g})  = \binom{n+d}{n}+\binom{n+d-2}{n-1},\quad n=1,2,3,\ldots.
$$
An orthogonal basis for $\CV_n(\VV_0^{d+1},\sw_{\b,\g})$ can be given via the Jacobi polynomials and 
spherical harmonics by parametrizing the integral over $\VV_0^{d+1}$ by 
\begin{equation}\label{eq:intV0}
   \int_{\VV_0^{d+1}} f(x,t) \d \sm(x,t) = \int_0^1 t^{\d-1} \int_{\sph} f(t\xi, t) \d \s(\xi)\d t. 
\end{equation}  
Let $\CH_m(\sph)$ be the space of spherical harmonics of degree $m$ in $d$ variables. Let 
$\{Y_\ell^m: 1 \le \ell \le \dim \CH_m(\sph)\}$ be an orthonormal basis of $\CH_m(\sph)$. Then the polynomials
\begin{equation*} 
  \sS_{m, \ell}^n (x,t) = P_{n-m}^{(2m + \b + d-1,\g)} (1-2t) Y_\ell^m (x), \quad 0 \le m \le n, \,\, 
      1 \le \ell \le \dim \CH_m(\sph),
\end{equation*}
consist of an orthogonal basis of $\CV_n(\VV_0^{d+1}, \sw_{\b,\g})$. 

The most interesting case is $\b = -1$, since the polynomials in $\CV_n(\VV_0^{d+1},\sw_{-1,\g})$ are eigenfunctions 
of a differential operator, $\Delta_{\g,0}$, defined by 
\begin{align}\label{differential operator}
\Delta_{\g,0} = \frac{1}{t^{d-2}(1-t)^\g} \frac{\partial}{\partial t} \left ( t^{d-1} (1-t)^{\g+1}  \frac{\partial}{\partial t}  \right)
    + t^{-1} \Delta_0^{(\xi)}. 
\end{align}
where $\Delta_0^{(\xi)}$ is the Laplace-Beltrami operator in $\xi \in \sph$ for $(x,t) = (t\xi,t)\in \CV_0^{d+1}$ and 
$\partial_t$ is the partial derivative in $t$ variable. More precisely, for $d \ge 2$, $\g> -1$ and $n =0,1,2,\ldots$,
\begin{equation}\label{eq:eigen-eqn}
    \Delta_{\g,0} u =  -n (n+\g+d-1) u, \qquad \forall u \in \CV_n(\VV_0^{d+1}, \sw_{-1,\g}).
\end{equation}
The relation is first proved in \cite[Theorem 7.2]{X20}. The spectral operator $\Delta_{\g,0}$ is rewritten in the present 
form in \cite{X23}.

The space $\CV_n(\VV_0^{d+1}, \sw_{\b,\g})$ satisfies an addition formula, which is a closed-form formula for
the reproducing kernel $\sP_n\big(\sw_{\b,\g}; \,\cdot, \cdot \big)$, 
$$
\sP_n\big(\sw_{-1,\g}; (x,t), (y,s)\big) = \sum_{k=0}^n \sum_{\ell=1}^{\dim \CH_k^{d+1}} \frac{ \sS_{k, \ell}^n (x,t)  \sS_{k, \ell}^n (y,s)}
   {\la \sS_{k, \ell}^n,\sS_{k, \ell}^n \ra_{\sw_{-1,\g}} },
$$
and the addition formula is of the simplest form when $\b = -1$ (\cite[Theorem 8.2]{X20}). Let $d \ge 2$ and $\g \ge -\f12$. 
Then, for $(x,t), (y,s) \in \VV_0^{d+1}$, 
\begin{align} \label{eq:sfPbCone}
 \sP_n \big(\sw_{-1,\g}; (x,t), (y,s)\big) =  b_{\g,d}  \int_{[-1,1]^2} & Z_{n}^{(\g+d-\f32,-\f12)} \big( 2 \zeta (x,t,y,s; v)^2-1\big) \\
  & \times  (1-v_1^2)^{\f{d-4}{2}} (1-v_2^2)^{\g-\f12} \d v, \notag
\end{align} 
where $b_{\g,d}$ is a constant so that $\sP_0\big(\sw_{-1,\g}; \cdot,\cdot\big) =1$, $Z_n^{(\a,\b)}$ is defined in
\eqref{eq:Zn}, and 
\begin{equation}\label{eq:zetaV0}
 \zeta (x,t,y,s; v)  = v_1 \sqrt{\tfrac{st + \la x,y \ra}2}+ v_2 \sqrt{1-t}\sqrt{1-s};
\end{equation}
moreover, the identity holds under limit when $\g = -\f12$ and/or $d = 2$. As is the case for spherical harmonics,
$\sP_n\big(\sw_{-1,\g})$ is the kernel of the orthogonal projection operator $\proj_n(\sw_{-1,\g}): L^2(\VV_0^{d+1},\sw_{-1,\g}) 
\to \CV_n(\VV_0^{d+1}, \sw_{-1,\g})$; that is, 
$$
\proj_n(\sw_{-1,\g};f) =  b_{\g,d} \int_{\VV_0^{d+1}} f(y,s) \sP_n\big(\sw_{-1,\g}; \,\cdot, (y,s) \big)  \sw_{-1,\g}(s) \d\sm(y,s).
$$

The Fourier orthogonal expansion of $f \in L^2(\VV_0^{d+1}, \sw_{-1,\g})$ is defined by 
$$
  L^2\left(\VV_0^{d+1}, \sw_{-1,\g}\right) = \sum_{n=0}^\infty  \CV_n(\VV_0^{d+1}, \sw_{-1,\g}): 
     \quad  f = \sum_{n=0}^\infty \proj_n(\sw_{-1,\g};f). 
$$
As an analog of \eqref{eq:multiplier_S}, we can define a mulitplier operatror $\sS_{\t, \sw_{-1,\g}}$ by
$$
\proj_n(\sw_{-1,\g}; \sS_{\t,\sw_{-1,\g}}f) = \frac{P_n^{(\l - \frac{1}{2}, -\frac{1}{2})} (\cos \t)}{P_n^{(\l - \frac{1}{2}, -\frac{1}{2})} (1)} \proj_n(\sw_{-1,\g}; f), \quad \l = \g+d-1, 
$$
and use the operator to define a modulus of smoothness  
\begin{equation} \label{eq:o_rHat}
 \widehat{\o}_r(f; t)_{p,\sw_{-1,\g}} = \sup_{0 \le \t \le t} \left\| \left(I - \sS_{\t,\sw_{-1,\g}}\right)^{r/2} f\right\|_{p,\sw_{-1,\g}}, 
   \quad 1 \le p \le \infty. 
\end{equation}
As a special case of the framework developed in \cite{X21}, the best approximation 
$$
   \sE_n(f)_{p, \sw_{-1,\g}} : = \sup_{Y \in \Pi_n(\VV_0^{d+1})} \left\| f- Y \right \|_{p, \sw_{-1,\g}}, \qquad 1 \le p \le \infty,
$$
is characterized by this modulus of smoothness with both direct and inverse theorems \cite[Theorem 3.12]{X21}.
Moreover, using the spectral operator $\Delta_{0,\g}$, we can define a $K$-functional by
\begin{equation}\label{eq:K_rHat}
 \widehat{\sK}_r(f, t)_{p,\sw_{-1,\g}} := \inf \left\{ \|f - g\|_{p,\sw_{-1,\g}} + 
      t^r \left\|(-\Delta_{\g,0})^{\frac{r}{2}} f \right\|_{p,\sw_{-1,\g}} \right\}
\end{equation}
where the infimum is taken over $g \in W_p^r(\VV_0^{d+1}, \sw_{-1,\g})$, which is shown to be equivalent to the 
modulus of smoothness in \cite[Theorem 3.13]{X21}, 
$$
   c_1  \widehat{\sK}_r(f, t)_{p,\sw_{-1,\g}} \le \widehat{\o}_r(f; t)_{p,\sw_{-1,\g}} \le c_2   \widehat{\sK}_r(f, t)_{p,\sw_{-1,\g}}.
$$

\subsection{Highly localized kernel on the conic surface} 
The framework in \cite{X21} is based on the existence of highly localized kernels, which will also play an essential
role in our study. For $\sw_{-1,\g}$ on the conic surface, such kernels are defined by
 \begin{equation} \label{def:Ln-kernel-V0} 
    \sL_n^\g \left( (x,t),(y,s)\right) = \sum_{k=0}^{\infty} \wh a \left(\frac{k}{n}\right) \sP_k(\sw_{-1,\g}; (x,t),(y,s)), 
\end{equation} 
where $\wh a$ is a cut-off function as defined in \eqref{eq:cut-off}. Let 
\begin{equation}\label{eq:w(n;t)}
     \sw_{\g,d} (n; t) =\big(t+n^{-2}\big)^{\f{d-2}{2}} \big(1-t+n^{-2}\big)^{\g+\f12}. 
\end{equation}
For any $\k > 0$ and \((x,t), (y,s) \in \VV_0^{d+1}\), the kernel $\sL_n(\sw_{-1,\g})$ satisfies the estimate \cite[Theorem 4.10]{X21},
\begin{equation}\label{V0-bound}
\left | \sL_n^\g\left( (x,t), (y,s) \right) \right|
	\le \frac{c_\kappa}{\sqrt{\sw_{\gamma,d}(n; t)} \sqrt{\sw_{\gamma,d}(n; s)}} G_{n,d}^\kappa \big( \sd_{\VV_0}((x,t), (y,s)) \big),
\end{equation}
where the function \(G_{n,d}^\kappa\) is defined, for future reference, by
\begin{equation} \label{eq:Gnd}
G_{n,d}^\kappa(u) = \frac{n^d}{(1 + n u)^\kappa}, \qquad u \in \RR,
\end{equation}
and $\sd_{\VV_0}(\cdot, \cdot)$ is the distant function on $\VV_0^{d+1}$ defined by \cite[Definiotn 4.1]{X21}
\begin{equation}\label{eq:distV0}
  \sd_{\VV_0} ((x,t), (y,s)): =  \arccos \left(\sqrt{\frac{\la x,y\ra + t s}{2}} + \sqrt{1-t}\sqrt{1-s}\right).
\end{equation}
Per this estimate, the kernel decays away from the diagonal faster than any polynomial rate, which is why it is called 
highly localized. Let us define an integral operator via the highly localized kernel,
\begin{align}\label{def:Ln-surface}
	 \sL_n^\g f(x,t) :=  \sb_\g \int_{\VV_0^{d+1}} f(y,s) \sL_n^\g \left(\sw_{-1,\g}; \big( (x,t),(y,s) \big)\right) \sw_{-1,\g}(s) \d \sm(y,s).
\end{align}
Then $\sL^\g_nf $ is a polynomial of near-best approximation since it satisfies the following theorem. 

\begin{thm}\label{thm:near-bestCS}
Let $f \in L^p\big(\VV_0^{d+1}, \sw_{-1,\g}\big)$ for $1 \le p < \infty$ or $f\in C\big(\VV_0^{d+1}\big)$ for $p = \infty$. Then
\begin{enumerate}[\rm (i)]
\item $\sL^\g_nf \in  \Pi_{2n}\big(\VV_0^{d+1}\big)$ and $\sL^\g_nf= f$ for $f \in \Pi_n\big(\VV_0^{d+1}\big)$.
\item For $n \in \NN$, $\|\sL^\g_n f\|_{p,\sw_{-1,\g}} \le c \|f\|_{p,\sw_{-1,\g}}$.
\item For $n \in \NN$, 
\begin{equation}\label{eqn-2.16}
\|f- \sL^\g_nf\|_{p,\sw_{-1,\g}} \le c\, \sE_n(f)_{p,\sw_{-1,\g}}, \qquad 1 \le p \le \infty.
\end{equation}
\end{enumerate}
\end{thm}

The proof of this theorem holds for analogous operators on spaces of homogeneous type \cite[Theorem 3.14]{X21}. 
We will need new properties of the highly localized kernels, stated in several lemmas below.  

The first one shows that the estimate for the highly localized kernel inherits the semi-product structure of the 
conic surface. Recall that the distance on $[0,1]$ is defined by
$$
\sd_{[0,1]}(s,t) = \arccos \left( \sqrt{s}\sqrt{t} + \sqrt{1-s} \sqrt{1-t} \right), \qquad s, t \in [0,1]. 
$$

\begin{lem} \label{lem:key_lem2}
For \(\k_1,\k_2 > 0\) and \((x,t), (y,s) \in \VV_0^{d+1}\), where \(x = t\xi\) and \(y = s\eta\) with \(\xi, \eta \in \sph\), 
$$
  \left| \sL_n^\g \big((t\xi,t),(s \eta,s) \big)\right| \le  \frac{c_\k}{\sqrt{\sw_{\g,d}(n;t)}\sqrt{\sw_{\g,d}(n;s)}} G_{n,1}^{\k_1}\!\left(\sd_{[0,1]}(s,t)\right) 
      G_{n,d-1}^{\k_2}\!\left(\sqrt{s} \sd_{\SS}(\xi,\eta)\right).  
$$
\end{lem}

\begin{proof}
It is shown in \cite[Proposition 4.3]{X21} that
\begin{equation}\label{eq:dist-sim}
  \sd_{\VV_0} \big((t\xi,t),(s\eta, s) \big) \sim \sd_{[0,1]}(s,t) + (s t)^{\f14} \sd_{\SS}(\xi,\eta).
\end{equation}
If $t \ge s$, then $(s t)^{\f14} \sd_{\SS}(\xi,\eta) \ge s^{\f12} \sd_{\SS}(\xi,\eta)$. If $t \le s$, then using the inequality 
$\sqrt{s} - \sqrt{t} \le \sd_{[0,1]}(s,t)$, we obtain
$$
    s^{\f14} - t^{\f14} \le \frac{1} {s^{\f14} + t^{\f14} } \sd_{[0,1]}(s,t),
$$
which implies 
\begin{align*}
  \sd_{[0,1]}(s,t) + (s t)^{\f14} \sd_{\SS}(\xi,\eta) \, &  \ge \sd_{[0,1]}(s,t) + \left( s^\f12 - \frac{s^{\f14}}{s^{\f14} + t^{\f14}}
     \sd_{[0,1]}(s,t) \right)  \sd_{\SS}(\xi,\eta)  \\
    &  \ge s^{\f12}\sd_{\SS}(\xi,\eta). 
\end{align*}
Consequently, by \eqref{eq:dist-sim}, we conclude that 
$$
   \sd_{\VV_0} \big((x,t),(y,s) \big) \ge \sd_{\VV_0}((s\xi,s),(s \eta,s)) \quad \hbox{and}\quad  
    \sd_{\VV_0} \big((x,t),(y,s) \big) \ge \sd_{[0,1]}(s, t). 
$$
Hence, it follows from the definition of $G_{n,d}^\k$ that
$$
  G_{n,d}^{\k_1+\k_2}\left( \sd_{\VV_0} \big((x,t),(y,s)\big) \right) \le c \,
  G_{n,1}^{\k_1}\left( \sd_{[0,1]} (s,t)\right) G_{n,d-1}^{\k_2}\left( \sqrt{s} \sd_{\SS}(\xi,\eta) \right)
$$
which completes the proof. 
\end{proof}

\begin{cor}\label{cor:key_int}
For $\k > 0$,  there is a constant $c_\k$ such that 
\begin{equation} \label{eq:key_int1}
  \int_0^1 t^{d-1} \frac{G_{n,1}^{\k_1}\left(\sd_{[0,1]}(s,t)\right) }{\sqrt{\sw_{\g,d}(n;t)}\sqrt{\sw_{\g,d}(n;s)}} \sw_{-1,\g}(t) \d t
     \le c_\k (s+n^{-2})^{\frac{d-1}{2}}.
\end{equation}
\end{cor}

\begin{proof}
In terms of the Jacobi weight $\varpi_{\a,\b}$ on $[0,1]$, $\sw_{-1,\g}(n;t) = \varpi_{\frac{d-3}{2},\g}(n;t)$. Hence, using the 
upper bound in the Lemma \ref{lem:key_lem2}, we see that the left-hand side of the stated inequality is bounded by 
\begin{align*}
  \int_0^1  \frac{1}{\sqrt{\varpi_{\frac{d-3}{2},\g}(n;t)}\sqrt{\varpi_{\frac{d-3}{2},\g}(n;s)}} & \frac{n}{(1+n \sd_{[0,1]}(s,t))^{\k_1}}
      \varpi_{d-2,\g}(t) \d t    \le c\, (s+n^{-2})^{\frac{d-1}{2}}, 
\end{align*}
by the proof of Lemma \ref{eq:int_Jacobi_ker} with $\g = d-2- \frac{d-3}{2} = \frac{d-1}{2}$ and $\delta = 0$, where we
have mapped $[-1,1]$ to $[0,1]$. 
\end{proof}

Let $L_n^{(\a,\b)}(\cdot,\cdot)$ be the highly localized kernel for the Jacobi weight defined in \eqref{eq:Ln-Jacobi}. 
We define the operator $ f \mapsto L_n^{(\b,\g)} f$ for $f\in L^1([0,1], \varpi_{\b,\g})$ by
$$
       L_n^{(\b,\g)} f(t) = c_{\b,\g} \int_0^1 f(s) \wt L_n^{(\b,\g)}(s,t) \varpi_{\b,\g}(s)\d s,
$$
where $\wt L_n^{(\b,\g)}(s,t) = L_n^{(\b,\g)}(1-2s, 1-2 t)$ is the highly localized kernel for $\varpi_{\b,\g}$. 

\begin{lem} \label{lem:LnSph=}
For $\g \ge 0$ and $n\in \NN$, 
$$
  \frac{1}{\s_d} \int_{\sph} \sL_n^\g  f (s\eta,s) \d \s(\eta) =  L_n^{(d-2,\g)} f(s).
$$ 
\end{lem}

\begin{proof}
By \eqref{def:Ln-kernel-V0} and the orthogonality of the spherical harmonics $Y_\ell^m$, we deduce
\begin{align*}
 \frac{1}{\s_d} \int_{\sph} \sL_n^\g \big((s\eta,s),(t\xi ,t)\big) \d \s(\eta)
      & = \sum_{k=0}^{2n} \wh a \left(\frac{k}{n}\right) \frac{P_{k}^{(d-2,\g)} (1-2t)P_{k}^{(d-2,\g)} (1-2s)}{h_k^{(d-2,\g)}} \\
       & =  L_n^{(d-2,\g)}(1-2s, 1- 2 t) = \wt L_n^{(d-2,\g)}(s,t),
\end{align*}
from which the desired identity follows from \eqref{eq:intV0}.
\end{proof}
 
The next lemma is about an analog of the operator $\sL_n^\g f$, which we define as 
\begin{equation}\label{eq:GnF}
 \sG_n^\g f(t\xi,t):=\sb_\g
    \int_{\VV_0^{d+1}} f(s\xi,s) \sL_n^\gamma \big((t\xi,t),(s\eta,s)\big) \sw_{-1,\g}(s) \d \sm(s\eta, s),
\end{equation}
in which the function $f$ has the variable $s \xi$ that is a mixture of components in $x = t\xi$ and $ y = s\eta$. 

\begin{lem} \label{lem:L_nG}
For $0 \le m \le n$, $\sG_n^\g  \sL_m^\g f = \sL_m^\g \sG_n^\g f  = \sL_m^\g f$. 
\end{lem}

\begin{proof}
Using \eqref{eq:intV0}, we obtain, by the proof of Lemma \ref{lem:LnSph=}, 
$$
    \sG_n^\g f(t\xi,t) = c_{d-2,\g} \int_0^1f(s\xi,s) \wt L_n^{(\d-2,\g)} (s,t) \sw_{d-2,\g}(s) \d s,
$$
where we have used $s^{d-1} \sw_{-1,\g}(s) = \sw_{d-2,\g}(s)$. Let $(z,u) = (u \zeta, u) \in \VV_0^{d+1}$. 
Then, the above identity shows immediately that 
\begin{align*}
  \sG_n^\g \sL_m^\g f(z, u)  = c_{d-2,\g} \int_0^1 L_m^\g f (s\zeta,s) \wt L_n^{(\d-2,\g)} (s,u) \sw_{d-2,\g}(s) \d s 
  = \sL_m^\g f(z, u),
\end{align*}
since $\sL_m^\g$ is a polynomial of degree $m \le n$ and $\wt L_n^{(\d-2,\g)} (\cdot,\cdot)$ is the 
kernel of the operator that reproduces polynomials of degree $n$ for $\sw_{d-2,\g}$. Similarly, 
\begin{align*}
  \sL_m^\g \sG_n^\g f(z, u) \, & = \sb_\g \int_{\VV_0^{d+1}} L_m^\g\big( (z,u), (t\xi,t)\big)  \sG_n^\g f(t\xi,t) \d \sm(t\xi,t) \\
       & = \sb_\g  \int_0^1 \int_{\sph} c_{d-2,\g} \int_0^1 \wt L_n^{(\d-2,\g)} (s,t) L_m^\g\big( (z,u), (t\xi,t)\big) \sw_{d-2,\g}(t) \d t\\
       & \qquad \qquad \qquad \times  f(s\xi,s)  \, \d \s(\xi) \, s^{d-1} \sw_{-1,\g}(s) \d s. 
\end{align*}
The intergral over $\d t$ is equal to $\sL_m^\g \big( (z,u), (s\xi,s)\big)$ since $\sL_m^\g\big( (z,u), (t\xi,t)\big)$ is a 
polynomial of degree $m$ in the variable $t$. Consequently, we obtain
$$
 \sL_m^\g \sG_n^\g f(z, u) = \sb_\g  \int_0^1 \int_{\sph} f(s\xi,s) L_m^\g\big( (z,u), (s\xi,s)\big)\d \s(\xi) \, s^{d-1} \sw_{-1,\g}(s) \d s,
$$ 
which is exactly $\sL_m^\g f(z,u)$.
\end{proof}

\section{Best approximation on the conic surface}
\setcounter{equation}{0}

We state and discuss the main result in the first subsection without getting into technical details.  
The properties of the modulus of smoothness are established in the second subsection. The proof of the
main theorem is given in the third to fifth subsections.

\subsection{Main result} 
Throughout this section, we denote the norm of $L^p(\VV_0^{d+1}, \sw_{-1,\g})$, $\g \ge 0$, by 
$$
   \| f\|_{p,\g} :=  \| f\|_{L^p(\VV_0^{d+1}, \sw_{-1,\g})}, \quad 1 \le p < \infty,
$$
and adopt the usual convention that $p=\infty$ is the uniform norm on $\VV_0^{d+1}$. 

Our new modulus of smoothness on $\VV_0^{d+1}$ is a fusion of the modulus of smoothness \eqref{eq:o_rSph} 
on the unit sphere and the Ditzian-Totik modulus of smoothness on the interval $[0,1]$. We define 
\begin{align*}
\left\| \Delta_{i,j,\t}^r f\right\|_{p,\g}^p := \int_{\VV_0^{d+1}} \left|  \Delta_{i,j,\t}^r f \right|^p \sw_{-1,\g}(t) \d \sm(x,t).
\end{align*}
Setting $x = t \xi$, $\xi \in \sph$, we will need to consider  $\|\Delta_{i,j,\frac{\t}{\sqrt{t}}}^r f \|_{p,\g}$ 
below. 

\begin{defn}\label{def:modulusV0}
For $f\in L^p(\VV_0^{d+1},  \sw_{-1,\g})$, $1\leq p \le \infty$, $t\in(0,1]$, and $r > 0$, the modulus of smoothness is defined by
\footnote{Strictly speaking, we should replace $\f \t {\sqrt{t}}$ by $\t \psi$ with $\psi(x,t):= \frac1{\sqrt{t}}$. We kept the current
notation since it is more transparent.}
\begin{align*}
  \o_r(f; h)_{p,\sw_{-1,\g}} := \,& \sup_{0 < \t \le h}   \max_{1\leq i < j \leq d }
  	\Big\| \Delta_{i,j, \frac{\t}{\sqrt{t}}}^r f\Big\|_{p,\g}+\sup_{0<\t \le h} 
	      \left\| \Delta_{\t\varphi}^r f_\xi \right\|_{L^p(\VV_{r h}, \sw_{-1,\g}) } 
\end{align*} 
where $f_\xi(t) = f(t \xi, t)$ and $ \VV_{r h}:= \left\{(x,t) \in \VV_0^{d+1}:  t\in J_{rh}\right\}.$ 
For $r=1$, we write $\o(f; h)_{p,\sw_{-1,\g}}=\o_1(f; h)_{p, \sw_{-1,\g}}.$
\end{defn} 
   
Since $\Delta_{i,j, \frac{\t}{\sqrt{t}}}^r$ is the forward difference acting on the angle $\t_{i,j}$ of the polar coordinates of 
$(x_i,x_j)$ plane, see \eqref{eq:Delta_ijt} for example, it is well-defined and so is this modulus of smoothness for 
all $r \in \NN$ and $1 \le p \le \infty$. We also define the $K$-functional using the derivatives for the Euler angles. 
 
\begin{defn}\label{eq:cs-K-func}
For $f\in L^p(\VV_0^{d+1}, \sw_{-1,\g})$, $1\leq p \le \infty$, and $r > 0$, the $K$-functional is defined by 
\begin{align*}
 \sK_r(f, h)_{p,\sw_{-1,\g}}: = \inf 
  	  \bigg \{ \|f-g\|_{p, \g}  & + h^r \left\|\varphi^r \partial^r_t g \right\|_{p, \g} 
	      + h^r  \max_{1\le i < j \le d} \left\| \f{1}{\sqrt{t}^r}D_{i,j}^r g\right\|_{p, \g} \bigg \},
\end{align*}
where the infimum is taken over all $g \in C^r(\VV_0^{d+1})$.
 \end{defn}  

For $r = 1, 2$, the $K$-functional is well defined for all $1 \le p\le \infty$, since $x = t\xi$ shows that 
$D_{i,j}^r = (x_i \partial_j - x_j \partial_i)^r$ contains a factor $t$, which dominates $\sqrt{t}^r$ in the 
denominator to warrant the boundedness of $\| \f{1}{\sqrt{t}^r}D_{i,j}^r g\|_{p, \g}$. For $r > 2$, the 
similar argument, with $D_{i,j}^r$ contains a factor $t$, shows that the $K$-functional is bounded if 
$$
    1 \le p < \frac{2(d-1)}{r -2}.
$$
Moreover, this inequality is sharp, as observed in \cite{X23}, since if $g_0(x,t) = x_i$, then 
$D_{i,j}^{2\ell} g_0(x,t) = (-1)^\ell x_i$ 
and $D_{i,j}^{2\ell+1} g_0(x,t) = (-1)^\ell x_j$, so that 
$$
  \left\| \f{1}{\sqrt{t}^r}D_{i,j}^r g_0 \right\|_{p, \g}^p \sim \int_0^1 \frac{t^{p +d-2}}{\sqrt{t}^{r p}} (1-t)^\g \d t
   \sim \int_0^1 t^{p+d-2-rp/2} \d t,
$$
which is finite for $p \ge 1$ if $r =1,2$ but only for $1 \le p < \frac{2(d-1)}{r -2}$ if $r > 2$.  

Our main result is a characterization for the quantity $\sE_n(f)_{p, \sw_{-1,\g}}$ by the new
modulus of smoothness or by the $K$-functional, which are equivalent for $r =1,2$.  

\begin{thm}\label{main-thmV0}
Let $\g \ge 0$ and let $f \in L^p(\VV_0^{d+1}, \sw_{-1,\g})$ if $1 \le p < \infty$ and 
$f\in C(\VV_0^{d+1})$ if $p = \infty$. Then for   $n =1,2,\ldots$, there hold 
\begin{enumerate} [\rm (i)]
\item  direct theorem:
$$
    \sE_n(f)_{p,\sw_{-1,\g}} \le c \, \o_r(f;  n^{-1})_{p,\sw_{-1,\g}};
$$
\item inverse theorem:   for $r =1, 2$  
$$
   \o_r(f; n^{-1})_{p,\sw_{-1,\g}} \le c \,n^{-r} \sum_{k=0}^n (k+1)^{r-1}\sE_k(f)_{p,\sw_{-1,\g}}.
$$
\item equivalence:      for $r =1, 2$  
$$
     c_1 \o_r(f; \rho)_{p,\sw_{-1,\g}} \le \sK_r(f; \rho)_{p,\sw_{-1,\g}} \le c_2 \o_r(f; \rho)_{p,\sw_{-1,\g}}.
$$
\end{enumerate}
\end{thm}

The proof will be carried out in the order of (i), (iii), (ii), since the proof of the inverse theorem is much 
easier to carry out in terms of the $K$-functional. 

\begin{rem}
We do not know if the inverse theorem, as stated above in terms of the modulus of smoothness, holds 
for $r > 2$. The inverse theorem requires a proof that bypasses the $K$-functional and works directly 
on the modulus of smoothness. 
\end{rem}

\subsection{Properties of modulus of smoothness} 
We show that the modulus of smoothness satisfies the usual properties, including the Marchaud-type inequality.  

\begin{prop}\label{prop-moduleV0}
Let $f\in L^p(\VV^{d+1}_0)$ for $1\leq p<\infty$ and $f\in C(\VV^{d+1}_0)$ for $p=\infty$. The modulus of 
smoothness $\o^r(f; h)_{p,\sw_{-1,\g}}$ satisfies the following properties:
\begin{enumerate} [\rm  (i)]
	\item For $0< h<h_0$, $  \o_r(f; h)_{p,\sw_{-1,\g}}\leq c \|f\|_{p,\g }$.
	\item For $\l >0$, $\o_r(f; \lambda h)_{p,\sw_{-1,\g}}\leq (\l+1)^r \o_r(f; h)_{p,\sw_{-1,\g}}$.
	\item (Marchaud inequality) For $0<h<\f12$, 
$$
    \o_r(f; h)_{p,\sw_{-1,\g}}\leq c  h^r \int_h^1\f{\o_{r+1}(f; u)_{p,\sw_{-1,\g}}}{u^{r+1}}  \d u .
$$
	\end{enumerate}
\end{prop}

Since the conic surface is a semi-product of the sphere with radius $t$ and the interval $[0,1]$ for $t$, 
the modulus of smoothness $\o_r(f;  n^{-1})_{p,\sw_{-1,\g}}$ is a fusion of the two moduli of smoothness, 
one on the sphere and the other on the interval. We split $\o_r(f; \delta)_{p, \sw_{-1,\g}}$ into two distinct 
parts accordingly,  
\begin{align}\label{modulusAB}    
\o_r(f; h)_{p, \sw_{-1,\g}} = \omega_A^r (f;h)_{p,\g} + \omega^r_B(f; h)_{p,\g},
\end{align}
where, setting $f_\xi(t) = f(t\xi,t)$, 
\begin{align*} 
\omega^r_A(f; h)_{p,\g} := &  \sup_{0<\t \le h} \left\| \Delta_{\t\varphi}^r f_\xi \right\|_{L^p( \VV_{r h},\sw_{-1,\g}) }   \\ 
\omega_B^r (f;\delta)_{p,\g} := & \sup_{|\t| \le \delta} \max_{1\le i<j\le d} \Big\| \Delta_{i,j, \frac{\t}{\sqrt{t}}}^r f\Big \|_{p,\sw_{-1,\g}}. 
\end{align*}
The proof of Proposition \ref{prop-moduleV0} follows from the following two lemmas. 
 
\begin{lem}\label{prop-moduleA}
	Let $f\in L^p(\VV^{d+1}_0)$ for $1\leq p<\infty$ and $f\in C(\VV^{d+1}_0)$ for $p=\infty$. The modulus of 
	smoothness $\o^r_A(f; h)_{p,\g}$ satisfies the following properties:
	\begin{enumerate} [\rm  (i)]
		\item For $0< h<h_0$, $  \o_A^r(f; h)_{p,\g}\leq c \|f\|_{p,\g }$.
		\item For $\l >0$, $\o_A^r(f; \lambda h)_{p,\g }\leq (\l+1)^r \o_A^r(f; h)_{p,\g }$.
		\item (Marchaud inequality) For $0<h<\f12$, 
	 $$
	\o_A^r(f; h)_{p,\g }\leq c  h^r \int_h^1\f{\o_A^{r+1}(f; u)_{p,\g}}{u^{r+1}}  \d u .
	$$
	\end{enumerate}
\end{lem}
 
\begin{proof}
The modulus of smoothness $\omega_A^r (f;h)_{p,\g}$ is defined via the integral 
\begin{equation*}
   \left\| \Delta_{\t\varphi}^r f_\xi \right\|_{L^p( \VV_{r h},\sw_{-1,\g}) }  
	=\int_{\sph}\left [ \int_{J_{r h}} t^{d-1} |\Delta^r_{\theta \varphi } f(t\xi,t)|^p \d\sigma(\xi)  \sw_{-1,\g}(t) \d t \right]  \d\sigma(\xi) . 
\end{equation*}   
Using $f_\xi(t)= f(t\xi, t)$ for $ t\in J_{rh}$, for (i), we apply the main-part estimate on the interval $[0,1]$ 
as stated in \cite[p.~58]{DT}, 
\[
	\|\Delta_{\t\varphi}^r f_\xi\|_{L^p(J_{rh},   \varpi_{d-2,\g})}^p \leq c\, \|f_{\xi}\|_{L^p((0,1),   \varpi_{d-2,\g})}^p,
\] 
and integrate both sides over $\xi \in \mathbb{S}^{d-1}$ to obtain the desired inequality in part~(i).
An analogous argument works for (ii), based on the result for the main-part moduli established in \cite[(6.2.8)]{DT}, 
which gives
$$
	 \sup_{|\theta| \leq \lambda h} \Big\|\Delta_{\theta \varphi}^r f_\xi\Big\|_{L^p(J_{rh},  \varpi_{d-2,\g})} 
	 \leq (\lambda + 1)^r \sup_{|\theta| \leq h} \Big\|\Delta_{\theta \varphi}^r f_\xi \Big\|_{L^p(J_{rh},  \varpi_{d-2,\g})}.
$$ 
Part (iii) is derived from the established Marchaud inequality for weighted main-part moduli given in 
\cite[Theorem 6.4.2]{DT}, which yields	 
$$
	\o_\varphi^r(f_\xi; h)_{p,\varpi_{d-2,\g} }\leq c  h^r \left(\int_h^1\f{\o_\varphi ^{r+1}(f_{\xi}; u)_{p,\varpi_{d-2,\g}}}{u^{r+1}}  \d u
	 + \|f_\xi\|_{L^p((0,1),   \varpi_{d-2,\g})}\right).
$$ 
The term $\|f_\xi\|_{L^p((0,1), \varpi_{d-2,\g})}$ is redundant and can be removed using the Jackson 
inequality \cite[(8.2.1)]{DT}.
\end{proof}
	
\begin{lem}\label{prop-module}
Let $f\in L^p(\VV^{d+1}_0)$ for $1\leq p<\infty$ and $f\in C(\VV^{d+1}_0)$ for $p=\infty$. The modulus of 
smoothness $\o^r_B(f; h)_{p,\g}$ satisfies the following properties:
\begin{enumerate} [\rm  (i)]
\item For $0< h <\delta_0$, $  \o_B^r(f; h)_{p,\g}\leq c \|f\|_{p,\g }$.
\item For $\l >0$, $\o_B^r(f; \lambda h)_{p,\g }\leq (\l+1)^r \o_B^r(f; h)_{p,\g }$.
\item (Marchaud inequality) For $0<h <\f12$, 
$$
  \o_B^r(f; h)_{p,\g }\leq c  h^r\int_h^1\f{\o_B^{r+1}(f; u)_{p,\g}}{u^{r+1}}  \d u .
$$
\end{enumerate}
\end{lem}

\begin{proof}  
The modulus of smoothness $\omega_B^r (f;h)_{p,\g}$ is defined via the integral 
\begin{equation*}
  \Big\|\Delta^r_{i,j,\f{\t}{\sqrt{t}}} f \Big\|^p_{p,\g}
     =\int_0^1 t^{d-1} \Big[\int_{\sph}|\Delta^r_{i,j,\f{\t}{\sqrt{t}}} f(t\xi,t)|^p \d\sigma(\xi)  \Big]\sw_{-1,\g}(t) \d t. 
\end{equation*}  
For $0 < t<1$, we denote by $f_t$ the function defined by $f_t(\xi) := f(t\xi, t)$, $\xi \in \sph$. 

For (i), we utilize the inequality on the unit sphere \cite[Lemma 4.2 (ii)]{DaiX} for $f_t$,
\begin{equation}\label{eqn-3.4}
\Big\| \Delta^r_{i,j,\frac{\theta}{\sqrt{t}}} f_t \Big\|_{L^p(\sph)} \leq 2^r \|f_t\|_{L^p(\sph)},
\end{equation}
and integrate it one more time in $t$. The same argument works for (ii) by using \cite[Lemma 4.2 (i)]{DaiX}, which gives
$$
  \sup_{|\theta| \leq \lambda h} \Big\|\Delta_{i,j,\frac{\theta}{\sqrt{t}}}^r f_t\Big\|_{L^p(\sph)} 
    \leq (\lambda + 1)^r \sup_{|\theta| \leq h} \Big\|\Delta_{i,j,\f{\t}{\sqrt{t}}}^r f_t\Big\|_{L^p(\sph)}.
$$
For (iii), we derive it from the corresponding result for trigonometric functions. To illustrate this, let us 
consider the case \((i, j) = (1, 2)\). By utilizing the identity  
\[
\Big \|\Delta^r_{1,2,\frac{\theta}{\sqrt{t}}} f_t \Big\|_{L^p(\sph)}^p = \int_0^1 z (1-z)^{\frac{d-4}{2}} \int_{\SS^{d-3}}
   \left[ \int_0^{2\pi} \Big|\overrightarrow{\Delta}^r_{\frac{\theta}{\sqrt{t}}} g_{z,y}(\phi)\Big|^p \, \d\phi \right] \d\s(y) \, \d z,
\]
where \(g_{z,y}(\phi) = f_t(z \cos\phi, z \sin\phi, \sqrt{1-z^2} y)\), $y\in\SS^{d-3}$, $z\in[0,1]$, and $\phi\in[0,2\pi]$, 
The result of (iii) follows from the Marchaud inequality for trigonometric functions and a change of variables.  
\end{proof}

The next two lemmas discuss the commutativity of the near-best approximation operator $\sL_n^\g f$ 
in \eqref{def:Ln-kernel-V0} with the difference operators $\Delta_{i,j,\f\t{\sqrt{t}}}$ and $\Delta_{\theta \varphi}$,
which will be needed in the proof of the direct theorem. 

\begin{lem} \label{lem:triV=Vtri2}
For $1 \le p \le \infty$,   
$$
 	\Delta_{\theta\varphi}^r \sL_n^\g f = \sL_n^\g \Big(\Delta_{\theta\varphi}^r f \Big ).
$$
In particular, for $\t > 0$,
$$
 	\o_A^r  \left(f- \sL_n^\g f, \t\right)_{p,\g} \le c\, \o_A^r (f,\t)_{p, \g}.
$$
\end{lem}
 
\begin{proof}
By the definitions of $\sL_n^\g f$ and the central difference,   
\begin{align*}
 \Delta_{h} \sL_n^\g f(t, \xi) \, & =  \sb_\g \int_{\VV_0^{d+1}} f(s\eta,s)
     \big[\sL_{n}^{\gamma}\big ( (s\eta,s),\big((t+\tfrac h2)\xi,(t+\tfrac h2)\big)\big)   \\
&  \qquad\qquad\qquad  - \sL_{n}^{\gamma}\left( (s\eta,s),\big((t-\tfrac h2)\xi,(t-\tfrac h2)\big)\right) \big] 
         \d \sigma(\eta)\sw_{-1,\gamma}(s) \d s \\
& =\sL_n^\g \left( f\big((t+\tfrac h2) \xi,t+\tfrac h2\big) - f\big((t-\tfrac h2)\xi, t-\tfrac h2\big)\right) =
	 \sL_n^\g  \Delta_h f(t, \xi),
\end{align*} 
which holds with $h= \t \varphi(t)$ as well and iteration gives the commutivity for $r > 1$. Moreover,
it follows by (ii) of Theorem \ref{thm:near-bestCS}, 
$$
  \left \|\Delta_{\theta \varphi }^r ( f - \sL_n^\g f)\right\|_{p,\gamma} =  
      \left\|\Delta_{\theta \varphi}^r f - \sL_n^\g  \Delta_{\theta \varphi}^r  f\right \|_{p,\gamma}
 	\le (1+c) \left \|\Delta_{\theta \varphi}^r f \right \|_{p,\gamma},
 $$
 from which the stated inequality follows.
 \end{proof} 

\begin{lem} \label{lem:triV=Vtri}
For $1 \le p \le \infty$, $1 \le i\neq j\le d$, and $t\in(0,1]$, 
$$
	\Delta_{i,j,\f{\t}{\sqrt{t}}}^r \sL_n^\g f =   \sL_n^\g \Big(\Delta_{i,j,\f{\t}{\sqrt{t}}}^r f \Big ).
$$
In particular, for $\t > 0$,
$$
	\o_B^r  \left(f- \sL_n^\g f, \t\right)_{p,\g} \le c\, \o_B^r (f,\t)_{p, \g}.
$$
\end{lem}

\begin{proof}
Let $T_Q f(x) := f(Q x)$ for $Q\in SO(d)$, where $SO(d)$ denotes  the special orthogonal group. By the definition of
$\sL_n^\g f$, we have
\begin{align*}
 T_Q \sL_n^\g f(t, \xi)  & = \int_{\VV_0^{d+1}} f(s\eta,s) \sL_{n}^{\gamma}\left( (s\eta,s),(tQ\xi,t)\right) \d \sigma(\eta)\sw_{-1,\gamma}(s) \d s \\
  & =\int_{\VV_0^{d+1}} f(s\eta,s) \sL_{n}^{\gamma}\left( (sQ^{-1}\eta,s),(t\xi,t)\right) \d \sigma(\eta)\sw_{-1,\gamma}(s) \d s \\
  & =\int_{\VV_0^{d+1}} f(sQ\eta,s) \sL_{n}^{\gamma}\left( (s\eta,s),(t\xi,t)\right) \d \sigma(\eta)\sw_{-1,\gamma}(s) \d s 
     =    \sL_n^\g (T_Qf)( t,  \xi)
\end{align*}
by the rotation invariance of $\d \s(y)$, which gives, by $\Delta_{i,j,\f{\theta}{\sqrt{t}}}^r = \Big(I -T_{Q_{i,j,\f{\t}{\sqrt{t}}}}\Big)^r$, the stated result. Thus, 
$$
 \Big \|\Delta_{i,j,\f{\theta}{\sqrt{t}}}^r ( f - \sL_n^\g f)\Big\|_{p,\gamma} =  \Big\|\Delta_{i,j,\f{\t}{\sqrt{t}}}^r f - \sL_n^\g  \Delta_{i,j,\f{\t}{\sqrt{t}}}^r  f\Big \|_{p,\gamma}
    \le (1+ c) \Big \|\Delta_{i,j,\f{\t}{\sqrt{t}}}^r f \Big \|_{p,\gamma},
$$
from which the stated inequality follows.
\end{proof} 

Combining the last two lemmas leads to, by \eqref{modulusAB}, the following corollary. 

\begin{cor} \label{cor:3.9}
For $1 \le p \le \infty$ and $\t > 0$, 
$$
\o_r  \left(f- \sL_n^\g f, \t\right)_{p,\sw_{-1,\g}} \le c\, \o_r (f,\t)_{p, \sw_{-1,\g}}.
$$
\end{cor}

\subsection{Proof of the direct theorem}
Let $\sL_n^\g f$ be the near-best approximation operator defined in \eqref{def:Ln-surface}. Since $\sL_n^\g f$ 
is a polynomial of degree $2n$, the direct theorem follows if we can establish the inequality 
$$
     \| f - L_{\lfloor \f n 2 \rfloor}^\g f\|_{p,\sw_{-1,\g}} \le c\, \o_r(f;  n^{-1})_{p,\sw_{-1,\g}}.
$$
To make use of the semi-product structure over the conic surface, we split $f - L_n^\g f$ into two parts. 
Since the kernel $\sL_n^\g((x,t),\cdot)$ has the unit integral over $\VV_0^{d+1}$, we can write 
\begin{align*}
 f(x,t) -  \sL_n^rf(x,t) =  \sb_\g \int_{\VV_0^{d+1}} \big[ f(x,t) -f(y,s)\big] \sL_n^\g \big((y,s),(x,t)\big) \sw_{-1,\g}(s)  \d \sm(y,s).
\end{align*}
With $x = t \xi$ and $y = s \eta$, we write the function $f(x,t) - f(y,s)$ as a sum  
$$
 f(x,t) -f(y,s) = \left[f(t\xi,t) -f(s \xi, s) \right]+ \left[ f(s \xi,s) -f(s \eta, s)\right], 
$$
and split $ f -  \sL_n^rf$ accordingly as 
\begin{equation}\label{eq:F1+F2}
   f(t\xi, t) -  \sL_n^rf(t\xi, t) = \sF_{n,1}^\g f (t,\xi) + \sF_{n,2}^\g f (t,\xi), 
\end{equation}
where the two terms on the right-hand side are defined via $\sG_n^\g$ in \eqref{eq:GnF} by
\begin{equation} \label{eq:Fni}
 \sF_{n,1}^\g f := f - \sG_n^\g f \quad \hbox{and}\quad \sF_{n,2}^\g f := \sG_n^\g f - \sL_n^\g f. 
\end{equation}
More explicitly, the two functions are given by 
\begin{align*}
 \sF_{n,1}^\g f (t,\xi) \,&= \sb_\g \int_{\VV_0^{d+1}} \big[ f(t\xi,t) -f(s\xi, s)\big] \sL_n^\g \big((s\eta,s),(t\xi ,t)\big) \sw_{-1,\g}(s) 
  \d \sm(y,s), \\
 \sF_{n,2}^\g f (t,\xi) \,& = \sb_\g \int_{\VV_0^{d+1}} \big[ f(s \xi,s) -f(s \eta, s)\big] \sL_n^\g \big((s\eta ,s),(t\xi,t)\big) \sw_{-1,\g}(s)  \d \sm(y,s).
\end{align*} 
While $\sF_{n,1}^\g$ contains the difference of $f(t\xi, t)$, for each fixed $\xi$, in $t$ variable, $\sF_{n,2}^\g$ contains the 
difference of $f(s\xi, s)$, for each fixed $s$, in $\xi$ variable. We deal with the two parts separately. 

By \eqref{eq:F1+F2} and \eqref{modulusAB}, the proof of the direct estimate (i) in Theorem \ref{main-thmV0} 
amounts to bounding the norm of $\sF_{n,1}^\gamma f$ by the modulus $\o_A^r(f, n^{-1})_{p,\g}$ and
the norm of $\sF_{n,1}^\gamma f$ by the moduolus $\o_B^r(f, n^{-1})_{p,\g}$, respectively. The proof is 
fairly long and is the content of the rest of this subsection. 

We start with the estimate for $\sF_{n,1}^\g$, which is relatively easy. 
 
\begin{prop}\label{prop:Fn1}
For $f \in L^p(\VV_0^{d+1}, \sw_{-1,\g})$, $1\le p \le \infty$ and $r \in \NN$, 
\begin{align*}
 \| \sF_{n,1}^\g f\|_{p,\sw_{-1,\g}}  \le c \, \omega^r_A (f;  n^{-1})_{p,\g}
\end{align*}
\end{prop}

\begin{proof} 
Parametrize the integral over $\d \sm(y,s)$ by $y = s\eta$ and \eqref{eq:intV0}, we obtain
\begin{align*}
\sF_{n,1}^\g f(t,\xi) 
  =  c_{d-2, \g} \int_0^1 \big[ f(t\xi,t) -f(s\xi, s)\big] \wt L_n^{(d-2,\g)}(s,t)  \varpi_{d-2,\g}(s) \d s 
\end{align*}
by Lemma \ref{lem:LnSph=}. Let $f_\xi(t) = f(t\xi,t)$. It then follows that 
$$
\sF_{n,1}^\g f (t,\xi) =  f_\xi (t) - L_n^{(d-2,\g)}f_\xi(t). 
$$
Since $L_n^{(d-2,\g)}f$ is the near-best approximation for $\varpi_{d-2,\g}$ on $[0,1]$, applying the direct theorem 
\eqref{eq:direct[0,1]} and \eqref{eq:K-o[0,1]} for the weight best approximation on $[0,1]$, we conclude 
$$
  \int_0^1 |\sF_{n,1}^\g f (t,\xi)|^p \varpi_{d-2,\g}(t) \d t \le c  \left[ E_n(f_\xi)_{p, \varpi_{d-2,\g}} \right]^p 
     \le c \left[ K_\varphi^r(f_\xi; n^{-1})_{ \varpi_{d-2,\g}} \right ]^p, 
$$
which implies immediately, since $\varpi_{d-2,\g}(t) = t^{d-1} \sw_{-1,\g}(t)$, that  
\begin{align} \label{eq:Fn1-tem}
  \|\sF_{n,1}^\g f \|_{p,\g}^p & = \sb_\g  \int_{\VV_0^{d+1}}  \left | \sF_{n,1}^\g f(x,t) \right |^p \sw_{-1,\g}(t) \d \sm(t,\xi) \\
 & = \frac{1}{\s_d} \int_{\sph} c_{\d-2,\g} \int_0^1 \left| \sF_{n,1}^\g f(t\xi,t)\right |^p \varpi_{d-2,\g}(t) \d t \d \s(\xi) \notag \\
 & \le c \int_{\sph}   \left[ \o_\varphi^{*r}(f_\xi; n^{-1})_{p,\varpi_{d-2,\g}} \right ]^p \d \s(\xi), \notag
\end{align}
where $\o_{\varphi}^{*r}(f)$ is the intermediate modulus in the proof of \cite[Theorem 6.1.1]{DT} that satisfies 
$$
K_\varphi^r(f;t)_{p, \sw} \le c\,\o_\varphi^{*r}(f;t)_{p, \sw} \le  c\,K_\varphi^{r}(f;t)_{p, \sw},
$$ 
and it is defined by, as a modification on $[0,1]$ of \cite[(6.1.9)]{DT} for the case $\RR_+$, 
\begin{align*}
 \o_{\varphi}^{*r} (f; v)_{p, \sw} = & \left\{ \frac1 v  \int_0^v \int_{\rho {r^2v^2}}^{1- \rho r^2 v^2}
      \left | \Delta_{\tau\varphi}^r f(t) \right |^p 
      \varpi_{d-2,\g}(t)\d t \d \tau \right \}^{1/p}  
\end{align*}
where $\rho =12$. Integrating the $p$-th power of the first term on the right-hand side of this modulus on $\sph$,
we obtain
\begin{align*}
  \int_{\sph} & \frac1 v  \int_0^v  \int_{\rho {r^2v^2}}^{1- \rho r^2 v^2} \left | \Delta_{\tau\varphi}^r f(t) \right |^p 
      \varpi_{d-2,\g}(t)\d t \d \tau  \d \s(\xi) \\
      &  \le  \frac1 v  \int_0^v \left\|  \Delta_{\tau \varphi}^r f_\xi \right \|_{L^p(\VV_{r h},\sw_{-1,\g}) }^p \d \tau
       \le   \sup_{0<\t \le h} \left\| \Delta_{\t\varphi}^r f_\xi \right\|_{L^p( \VV_{r h},\sw_{-1,\g})}.
 \end{align*}
  Consequently, by the 
definition of $\omega_A^r (f; h)_{p,\g}$, we conclude from \eqref{eq:Fn1-tem} that $\|\sF_{n,1}^\g\|_{p,\g}$ is bounded by 
$\omega_A^r (f; n^{-1})_{p,\g}$. The proof is completed. 
\end{proof}
 
\begin{rem} \label{rem:o_*r}
The modulus $\o_{\varphi}^{r} (f; v)_{p, \sw}$ and  $\o_{\varphi}^{*r} (f; v)_{p, \sw}$ are equivalent. The latter will be used
whenever we need to take an integral of $\o_{\varphi}^{r} (f_\xi; v)_{p, \sw}$ for the $\xi$ variable. 
\end{rem}
 
The proof of the boundedness for $\sF_{n,2}^\g$ is much more involved. The key step is the following
lemma. 

\begin{lem} \label{lem:key_lem1}
For $f \in L^p(\VV_0^{d+1})$, 
\begin{align*} 
  \int_{\VV_0^{d+1}} & \int_{\VV_0^{d+1}}  
      |f(s\xi,s) - f(s\eta,s)|^p |\sL_{n}^{\gamma}\left( (x,t),(y,s)\right)| \sw_{-1,\gamma}(t) \d \sm(x,t) 
        \sw_{-1,\gamma}(s) \d \sm(y,s) \notag \\
    & \le c \sup_{\t \le n^{-1}} \max_{1 \le i,j\le d} \int_{\VV_0^{d+1}} \left| \Delta_{i,j,\frac{\t}{\sqrt{s}}}f(s\xi,s)\right|^p \d \s(\xi)\sw_{-1,\g}(s)\d s
     = c\, \omega_B \left(f; n^{-1}\right)_{p,\g}.  
\end{align*}
\end{lem}

\begin{proof}
Let $I$ denote the left-hand side of the stated inequality. For $\delta > 0$, let $\sc_\delta(x,t)$ be the ball
centered at $(x,t)$ with radius $\delta$, 
$$\sc_\delta(x,t):= \{(y,s): \sd_{\VV_0}((x,t), (s,y))\le \delta\}.
$$
We write $I$ as a sum of two, $I = A+A^c$, by splitting the inner integral, so that 
$$
   I = A + A^c = \int_{\VV_0^{d+1}}  \int_{\sc_\delta (x,t)} \cdots \d\sm(y,s)  \d \sm(x,t)+ 
      \int_{\VV_0^{d+1}} \int_{\sc_\delta (x,t)^c} \cdots \d\sm(y,s)   \d \sm(x,t),
$$
where $\sc_\delta (x,t)^c$ denotes the complement of $\sc_\delta (x,t)$. Using the estimate \eqref{V0-bound}
and, by $G_{n,d}^{\k_1+\k_2} =  n^{-d} G_{n,d}^{\k_1} G_{n,d}^{\k_2}$ and 
$G_{n,d}^{\k_1} ( \sd_{\VV_0}((x,t), (y,s))) \ge c n^{d-\k_1}$ for $(y,s) \in \sc_\delta (x,t)^c$,  we obtain by
Lemma \ref{lem:key_lem2}  
\begin{align*}
  G_{n,d}^\k \big( \sd_{\VV_0}((x,t), (y,s)) \big) \le c\, n^{-\k_1} G_{n,1}^{\k_2}(s,t)
\end{align*}
for $(y,s) \in \sc_\delta (x,t)^c$. Consequently, we obtain 
\begin{align*}
A^c \, & \le  c n^{-\k_1}  \int_{\VV_0^{d+1}} \int_{\VV_{0}^{d+1}} 
            |f(s\xi,s) - f(s\eta,s)|^p  \frac{G_{n,1}^{\k_2}(s,t)}{ \sqrt{ \sw_{d,\g}(n;s)}\sqrt{\sw_{d,\g}(n;t)}} \\
    & \qquad\qquad \times  \sw_{-1,\gamma}(s) \sw_{-1,\gamma}(t)\d \sm(y,s) \d \sm(x,t) \\
     & \le  c n^{-\k_1} 
   \int_{\sph} \int_{\VV_{0}^{d+1}}  |f(s\xi,s) - f(s\eta,s)|^p (s+n^{-2})^{\f{d-1}{2}}   \sw_{-1,\gamma}(s) \d \sm(y,s) \d \s(\xi),
\end{align*}
where the second step follows, since $|f(s\xi,s) - f(s\eta,s)|^p$ is independent of $t$, from \eqref{eq:key_int1}. 
This implies, in particular, that if we choose $\k = p$, then 
$$
  A^c \le c_\k n^{-p} \int_{0}^1s^{d-1} \left[ \int_{\sph} \int_{\sph}  |f(s\xi,s) - f(s\eta,s)|^p  \d \s(\eta) \d \s(\xi) \right] \sw_{-1,\gamma}(s) \d s.
$$ 
For the double integral on the sphere, we let $f_s (\xi) = f(s \xi,s)$ and follow \cite[p. 90]{DaiX} to conclude that 
\begin{align*}
  A^c \, & \le c_\k n^{-\k+p} \int_0^1 s^{d-1} \max_{i, j} \sup_{0< \t<n^{-1}} \left\| \Delta_{i,j,\t} f_s \right\|_p^p \sw_{-1,\g}(s) \d s \\
   & \le c_\k   \int_0^1 s^{d-1} \max_{i, j} \sup_{0< \t<n^{-1}} \left\| \Delta_{i,j,\frac{\t}{\sqrt{s}}} f_s \right\|_p^p \sw_{-1,\g}(s) \d s \\
    & \le c_\k  \int_0^1 s^{d-1} \left[n \int_0^{n^{-1}} \max_{i,j} \int_\sph |\Delta_{i,j,\frac{\t}{\sqrt{s}}} 
         f_s(\xi)|^p \d \s(\xi) \d \t \right]  \sw_{-1,\g}(s) \d s \\
    & \le c_\k   n \int_0^{n^{-1}} \omega_B(f;n^{-1}) \d \t \le c_\k \omega_B(f; n^{-1})_{p,\g},
\end{align*}
where we used \cite[Lemma 4.2.2 (iv)]{DaiX} in the third step. 

We now consider the estimate for $A$, which we need to break into two cases. Let 
$$
   \VV_{0,n}^{d+1}:= \{(s \eta,s): \eta \in \sph, \,\, n^{-2} \le s \le 1\}, \quad n = 2, 3, \ldots 
$$
and denote by $(\VV_{0,n}^{d+1})^c$ its complement in $\VV_0^{d+1}$. We write 
$$
 A = B+ B^c = \int_{\VV_{0,n}^{d+1}}   \int_{\VV_0^{d+1}} \cdots   \d \sm(x,t)\d\sm(y,s) + 
 \int_{(\VV_{0,n}^{d+1})^c}   \int_{\VV_0^{d+1}} \cdots   \d \sm(x,t)\d\sm(y,s).
$$

Since $(y,s) \in \sc_\delta(x,t)$ implies, by \eqref{eq:dist-sim},
$\sqrt{s} \sd_{\SS}(\xi,\eta) \le c \delta$ or $\eta$ belongs to the spherical cap $\sc_\SS(\eta, \delta'/\sqrt{s})$ 
of $\sph$, where $\delta' = c \delta$. Writing $\delta'  =\delta$ for convenience and let $\delta_s:= \min\{\f1{8 d}, \frac{\delta}{\sqrt{s}}\}$, we then obtain
\begin{align*}
  B & \le c \int_{\VV_0^{d+1}} \int_{n^{-2}}^1 s^{d-1} \int_{\sc_\SS(\xi, \delta_s)} |f(s\xi,s) - f(s\eta,s)|^p 
       |\sL_{n}^{\gamma}\left( (x,t),(y,s)\right)|\d \s(\eta)     \\ 
            &  \qquad\qquad\qquad\qquad \times   \sw_{-1,\gamma}(s) \d s \sw_{-1,\gamma}(t) \d \sm(x,t) \\
    &  \le c \int_{n^{-2}}^1 s^{d-1} \Lambda^{\k_2}(s) s^{d-1} \sw_{-1,\g}(s) \d s     
\end{align*}
by Lemme \ref{lem:key_lem2} and Corollary \ref{cor:key_int}, where $\Lambda^\k(s)$ is defined by 
$$
  \Lambda^\k(s)= \int_\sph  \int_{\sc_\SS(\xi, \delta_s)} |f(s\xi,s) - f(s\eta,s)|^p \frac{(1+n \sqrt{s})^{d-1}}
         {(1+ n \sqrt{s} \sd_\SS(\xi,\eta))^\k}  \d \s(\eta)  \d \s(\xi).
$$

The integral $\Lambda^\k(s)$ can be estimated by following the scheme of the proof of \cite[Lemma 4.3.1]{DaiX}, which 
utilizes several reductions based on symmetry and spherical geometry that can be followed verbatim. By symmetry,
the first reduction shows that it is sufficient to consider $\Lambda_d^\k(s)$, which is $\Lambda^\k(s)$ with $\sph$ replaced 
by the set
$$
E_d^+ = \left\{\xi = (\xi', \xi_d) \in \sph: |\xi_d| \ge \frac{1}{\sqrt{d}}, \,\, \xi' \in \BB^{d-1}\right\},
$$
and, following the proof in \cite[p. 90]{DaiX}, it follows that 
\begin{align*}
 \Lambda_d^\k(s)\,& =  \int_{E_d^+}  \sph \int_{\sc_\SS(\xi, \delta_s)} |f(s\xi,s) - f(s\eta,s)|^p \frac{(1+n \sqrt{s})^{d-1}}
         {(1+ n \sqrt{s} \sd_\SS(\xi,\eta))^\k}  \d \s(\eta)  \d \s(\xi) \\
   &\le \int_{\|\xi'\|\le d^*}  \int_{\sc_\SS(\xi, \delta_s)} |f(s\xi,s) - f(s\eta,s)|^p \frac{(1+n \sqrt{s})^{d-1}}
         {(1+ n \sqrt{s} \sd_\SS(\xi,\eta))^\k}  \d \s(\eta) \frac{\d \xi'}{\sqrt{1-\|\xi'\|^2}},
\end{align*}
where $d^* = \sqrt{1-d^{-1}}$. Since $\xi_d \ge \frac 1 {\sqrt{d}}$ and, for $\eta\in \sc_\SS(\xi,\delta_s)$, 
$\sd_\SS(\xi,\eta) \le \delta_s \le \frac{1}{8d}$, we further deduce by the proof of \cite[p. 90--91]{DaiX} that
$\Lambda_d^\k(s)$ is the bounded by a sum of integrals 
$
  \Lambda_d^\k(s) \le c \sum_{j=1}^{d-1} A_j(s),
$
where $A_j(s)$ is given by, setting $g_s(\xi) = f(s (\xi', \sqrt{1-\|\xi'}), s)$,   
$$
 A_j(s)=  \int_{\|\xi'\|\le d^*+\delta} \int_{\|\ub\|\le \delta_s} |g_s(\xi') - g_s(\eta' + u_j \eb_j)|^p \frac{(1+n \sqrt{s})^{d-1}}
         {(1+ n \sqrt{s} \|\ub\|)^\k}  \d \ub \d \xi', 
$$
where $\eb_j$ denotes the $j$-th coordinate vector. It is sufficient to consider $A_1$. Making a change of variable
$\ub \mapsto \vb /\sqrt{s}$, we obtain
\begin{align*}
 A_1(s) \,& \le \int_{\|\xi'\|\le d^*+\delta} \int_{\sqrt{s} \|\vb\|\le \delta_s} \left |g_s(\xi') - g_s\left(\eta' + \tfrac{u_1}{\sqrt{s}} \eb_j\right)\right|^p 
      \frac{(1+n\sqrt{s})^{d-1}} {(1+ n \|\ub\|)^\k}  \frac{\d \vb}{(\sqrt{s})^{d-1}} \d \xi' \\ 
 & \le \int_{\|\xi'\|\le d^*+\delta} \int_{\|\vb\| \le \delta} \left |g_s(\xi') - g_s\left(\eta' + \tfrac{v_1}{\sqrt{s}} \eb_j\right)\right|^p 
      \frac{1} {(1+ n |v_1|)^\k} \d v_1 \d \xi' \\
   & \le \int_{\|\xi'\|\le d^*+\delta} \int_{-\delta}^\delta \left |g_s(\xi') - g_s\left(\eta' + \tfrac{v_1}{\sqrt{s}} \eb_j\right)\right|^p 
      \frac{n} {(1+ n |v_1|)^{\k-d-1}} \d v_1 \d \xi  
\end{align*} 
where the second step uses $1+n \sqrt{s} \sim 2 n \sqrt{s}$ for $s \ge n^{-2}$ and the third step follows as that
of \cite[(4.3.4)]{DaiX}. Apart from $1/\sqrt{s}$ associated with $v_1$, the last integral is precisely \cite[(4.3.4)]{DaiX}, 
so that we can follow the rest of the proof on \cite[p. 92]{DaiX} to conclude that
\begin{align*}
  A_1(s) \, & \le c  \int_{- \delta^*}^{\delta^*}  \sup_{|\t| \le |v|} \left \|\Delta_{1,d,\frac{\t}{\sqrt{s}}} g_s \right\|_{L^p(\sph)}^p 
     \frac{n} {(1+ n |v|)^{\k-d-1}} \d v \\
  & \le c \sup_{|\t| \le n^{-1} }\left \|\Delta_{1,d,\frac{\t}{\sqrt{s}}} g_s \right\|_{L^p(\sph)}^p  
      \sim c  n  \int_0^{\frac{1}{n} }\left \|\Delta_{1,d,\frac{\t}{\sqrt{s}}} g_s \right\|_{L^p(\sph)}^p  \d \t,
\end{align*}
where the last step follows from \cite[Lemma 4.2.2 (iv)]{DaiX}. This leads to an upper bound of $\Lambda^\k(s)$, which
gives then
$$
 B \le c \int_0^1 s^{d-1} \Lambda^{\k_2}(s) s^{d-1} \sw_{-1,\g}(s) \d s 
     \le c  { n \int_0^{\frac{1}{n}} } \omega_B(f; n^{-1})_{p,\g} \d r \le c \omega_B(f; n^{-1})_{p,\g}.
$$
This completes the estimate for $B$. Next, we consider $B^c$, which integrates over $(\VV_0^{d+1})^c$ in
$(y,s)$ variables. 

Since $\sw_{\g,d}(n,t) \sim \varpi_{d-3,\g}(n, 1-2t)$, it follows from Lemma \ref{lem:key_lem2}, 
 \eqref{eq:key_int1} and $|G_{n,d-1}^{\k_2}\left( \sqrt{s} \sd_{\SS}(\xi,\eta) \right)| \le n^{d-1}$, 
$$
  \int_0^1 \frac{t^{d-1}}{\sqrt{\sw_{\g,d}(n,s)}\sqrt{\sw_{\g,d}(n,t)}}
     |L_n^\g((x,t),(y,s)|  \sw_{-1,\g}(t) \d t \le c n^{d-1} (s+n^{-2})^{\f{d-1}{2}} \le c
$$
for $0\le s \le n^{-2}$. Hence, by \eqref{V0-bound}, we obtain an upper bound of $B^c$ given by
\begin{align*}
   B^c \le  &   \int_0^{n^{-2}} s^{d-1} 
    \left[\int_{\sph}  \int_{\sph} |f(s\xi, s) - f(s \eta,s)|^p)\d(\eta) \d \s(\xi) \right]
       \sw_{-1,\g}(s) \d s \d t \\
   \, & \le   \int_0^{n^{-2}}s^{d-1} \int_{\sph} \max_{i<j} \sup_{0\le \t\le 2 \pi} \left\| \Delta_{i,j,\t} f_s \right\|_p^p
    \sw_{-1,\g}(s) \d s,
\end{align*}
where the second step follows as in the estimate of $A^c$. Now, using 
$$
 \sup_{0\le \t\le 2 \pi} \left\| \Delta_{i,j,\t} f_s \right\|_p 
  = \sup_{0\le \t\le 2 \pi} \Big\| \Delta_{i,j,\frac{\t \sqrt{s}}{\sqrt{s}}} f_s \Big\|_p
   \le c (1+n \sqrt{s})  \sup_{0\le \t\le n^{-1}} \left\| \Delta_{i,j,\frac{\t}{\sqrt{s}}} f_s \right\|_p 
$$
and $n \sqrt{s} \le 1$ for $0 \le s \le n^{-2}$, we obtain the desired upper bound for $B^c$. 
This completes the proof. 
\end{proof}
  
Next, we proceed to analyze the estimation of  $\sF_{n,2}^\g$.

\begin{prop}\label{prop:Fn2}
For $f \in L^p(\VV_0^{d+1}, \sw_{-1,\g})$, $1\le p \le \infty$, 
\begin{align*}
 \| \sF_{n,2}^\g f \|_{p,\g}  \le c \, \omega_B \big(f; n^{-1}\big)_{p,\g}.
\end{align*}
\end{prop}

\begin{proof}
When $1\leq p < \infty$, this follows immediately from Lemma \ref{lem:key_lem1}. We now
consider the case $p = \infty$, for which we have 
\begin{align*}
  |f(s\xi,s)-f(s\eta,s)|  \leq \omega_{\SS } (f_s, \sd_{\SS} (\xi,\eta))_{\infty} \, &
	 = \sup_{|\theta|\leq  \sd_{\SS} (\xi,\eta)} \max_{i,j}\|\Delta_{i,j,\theta} f_s\|_{\infty}  \\
	& \leq \max_{i,j}\frac{1}{ \sd_{\SS} (\xi,\eta)}\int_0^{ \sd_{\SS} (\xi,\eta)}  |\Delta_{i,j,\theta} f_s| \d \theta, 
\end{align*} 
where the second inequality follows from   \cite[Lemma 2.6(iv)]{DaiX2}. Replacing $\theta $ by $\f{\theta}{\sqrt{s}}$, we have 
\begin{align*}
	\max_{i,j}\frac{1}{ \sd_{\SS} (\xi,\eta)}\int_0^{ \sd_{\SS} (\xi,\eta)} &  |\Delta_{i,j,\theta} f_s| \d \theta 
  =   \max_{i,j}\frac{1}{ \sqrt{s}\sd_{\SS} (\xi,\eta)}\int_0^{ \sqrt{s}\sd_{\SS} (\xi,\eta)}  |\Delta_{i,j,\f{\theta}{\sqrt{s}}} f_s| \d \theta \\
 & \leq \frac{1}{ \sqrt{s}\sd_{\SS} (\xi,\eta)}\int_0^{ \sqrt{s}\sd_{\SS} (\xi,\eta)} 
      \omega_{B}(f, \sqrt{s}\sd_{\SS} (\xi,\eta))_{\infty} \d \theta
 \\& \leq \frac{1}{ \sqrt{s}\sd_{\SS} (\xi,\eta)}\int_0^{ \sqrt{s}\sd_{\SS} (\xi,\eta)} (1+n\sqrt{s}\sd_{\SS} (\xi,\eta))
    \omega_{B}(f, n^{-1})_{\infty} \d \theta  \\& \leq c(1+n \sqrt{s}\sd_{\SS} (\xi,\eta)) \omega_{B}(f, n^{-1})_{\infty}.  
\end{align*}
Consequently, by the   definition of $\sF_{n,2}^\g$ and Lemma \ref{lem:key_lem2}, we have  
\begin{align*}
	& \left | \sF_{n,2}^\gamma(t,\xi) \right | 
	 \le c  \omega_{B}(f, n^{-1})_{\infty}  \int_0^1s^{d-1}  \int_{\SS^{d-1}}  (1+n \sqrt{s}\sd_{\SS} (\xi,\eta)) \\
	 &\qquad  \times \frac{ |G_{n,1}^{\k_1}(\d_{[0,1]}(s,t))| }{\sqrt{\sw_{-1,\g}(n; s)} \sqrt{\sw_{-1,\g}(n,t)}}  
	 \left |G_{n,d-1}^{\k_2}\left( \sqrt{s}\sd_{\SS} ((s\xi,s), (s \eta, s))\right) \right |  \d \s(\eta) \sw_{-1,\g}(s)\d s \\
	& \qquad\le c   \int_0^1\f{s^{d-1}  \omega_{B} (f , n^{-1})_{\infty }  }{(s+n^{-2})^{\f{d-1}{2}}} 
	   \frac{ |G_{n,1}^{\k_1}(\d_{[0,1]}(s,t))| }{\sqrt{\sw_{-1,\g}(n; s)} \sqrt{\sw_{-1,\g}(n,t)}}  \sw_{-1,\g}(s)\d s  \\ 
	 & \qquad   \leq c  \omega_{B} (f , n^{-1})_{\infty },
\end{align*} 
where  the last  inequality follows \eqref{eq:key_int1}.  Hence, we have 
$$  
\left \| \sF_{n,2}^\gamma(t,\xi) \right \|_\infty  
=\sup_{(\xi,t)\in\VV_0^{d+1}} |\sF_{n,2}^\gamma(t,\xi)| \leq c  \omega_{B} (f , n^{-1})_{\infty }.  
$$
This completes the proof. 
\end{proof}

We are ready to prove the direct theorem in Theorem \ref{main-thmV0}. The proof uses induction on $r$ by applying
Marchaud's inequality, a fairly standard procedure by now; see, for example, \cite[p. 94]{DaiX} for the proof on the 
unit sphere.  

\begin{proof}[Proof of (i) Theorem \ref{main-thmV0}]
For $r=1$, the direct estimate follows immediately, by \eqref{eq:F1+F2}, from Proposition \ref{prop:Fn1} and 
Proposition \ref{prop:Fn2}. 

For $r>1$, the proof uses induction. Assume we have proven the direct estimate for some $r \in \NN$. 
Let $g = f - \sL_{\lfloor \frac n 2 \rfloor}^\g f$. Since $\sL_m \sL_n = \sL_m$ for $m < n$, it follows that
$\sL^\g_{\lfloor \frac n 4 \rfloor} g = 0$. Hence, by \eqref{eqn-2.16} and the induction hypothesis, we obtain 
\begin{align*}
	\|g\|_{p,\g} = \left \|g - \sL^\g_{\lfloor \frac n 4 \rfloor} g\right \|_{p;\g} \le c\, \o_r(g; n^{-1})_{p,\sw_{-1,\g}}. 
\end{align*}
Applying parts (i)--(iii) of Proposition~\ref{prop-moduleV0}, we obtain, for every \( m \in \mathbb{N} \),
\begin{align*}
  \omega_r(g; n^{-1})_{p,\sw_{-1,\g}} 
	& \le c n^{-r} \int_{n^{-1}}^{2^m n^{-1}} \frac{\omega_{r+1}(g; u)_{p,\sw_{-1,\g}}} {u^{r+1}}\, \mathrm{d} u
	+ c_r n^{-r}\|g\|_{p,\g} \int_{2^{m}n^{-1}}^1 \frac{\mathrm{d} u}{u^{r+1}}  \\
	&\le c_{m}\, \omega_{r+1}(g, n^{-1})_{p,\sw_{-1,\g}} + c^{\prime} 2^{-mr} \|g\|_{p,\g},
\end{align*} 
where the constant \( c^{\prime} \) is independent of \( m \).  
Choose $m$ such that \( \frac{1}{4} \le  c  c^{\prime} 2^{-mr} < \frac{1}{2} \). Then the above two inequalities 
lead to 
\begin{align*}
	\|g\|_{p,\g} \le c\,  \omega_{r+1}(g; n^{-1})_{p,\sw_{-1,\g}} = 
	    c\, \omega_{r+1}\Big( f - \sL_{\lfloor \frac n 2 \rfloor}^\g f; n^{-1}\Big)_{p,\sw_{-1,\g}}, 
\end{align*}
so that, by applying Corollary \ref{cor:3.9}, we finally obtain 
\begin{align*}
	\left\| f - \sL_{\lfloor \frac n 2 \rfloor}^\g f \right\|_{p,\g} = \|g\|_{p,\g} \le c\, \o_{r+1}(f; n^{-1})_{p,\sw_{-1,\g}},
\end{align*}
which completes the proof.
\end{proof}

\subsection{Proof of equivalence of modulus of smoothness and $K$-functional}
We follow the approach for proving the equivalence on the unit sphere \cite[Theorems 6.10 and 6.14]{DaiX2}.
We first consider the lower estimate $ \o_r(f; \rho)_{p,\sw_{-1,\g}}\leq c  \sK_r(f; \rho)_{p,\sw_{-1,\g}}$, where $r=1,2$. 
Let $g_\rho\in C^r(\VV_0^{d+1})$ be chosen to satisfy 
\begin{align*}
    \|f-g_\rho\|_{p,\g}&\leq c \sK_r(f; \rho)_{p,\sw_{-1,\g}};\\
    \rho^r\max_{1\leq i\leq j\leq d}\|D_{i,j}^r g_{\rho}\|_{p,\g}&\leq c \sK_r(f; \rho)_{p,\sw_{-1,\g}};\\
      \rho^r \|\varphi^r\partial_t^r g_\rho\|_{p,\g}&\leq c \sK_r(f; \rho)_{p,\sw_{-1,\g}}.
\end{align*}
By the triangle inequality, we obtain
$$
\o_r(f; \rho)_{p,\sw_{-1,\g}} \leq \o_r(f-g_\rho; \rho)_{p,\sw_{-1,\g}}+ \o_r(g_\rho; \rho)_{p,\sw_{-1,\g}},
$$
so that it is sufficient to establish the inequalities 
\begin{align} \label{eq:o_r<K_r}
\o_r(f-g_\rho; \rho)_{p,\sw_{-1,\g}}\leq c\,  \sK_r(f; \rho)_{p,\sw_{-1,\g}} \quad \hbox{and} \quad 
\o_r(g_\rho; \rho)_{p,\sw_{-1,\g}} \le c\,\sK_r(f; \rho)_{p,\sw_{-1,\g}}.
\end{align}
To establish these inequalities, it is essential to demonstrate that for \( g \in C^r(\VV_0^{d+1}) \), the following hold:
 \begin{align*} 
 	\Big \|\Delta_{i,j,\f{\t}{\sqrt{t}}}^r (f-g)\Big \|_{p,\g} &\leq c \|f-g\|_{p,\g}; \qquad 
	\Big\|\Delta_{i,j,\f{\t}{\sqrt{t}}}^r g\Big\|_{p,\g} \leq c \t^r \Big \| \f{1}{\sqrt{t}^r} D_{i,j}^r g \Big\|_{p,\g};\\
	\left\| \Delta_{\t\varphi}^r (f-g)\right\|_{L^p(\VV_{r\rho},\sw) }  & \leq c \|f-g\|_{p,\g}; \qquad
	\left\| \Delta_{\t\varphi}^r g\right\|_{L^p(\VV_{r\rho},\sw) } \leq c \t^{r}\|\varphi^r \partial_t^r g\|_{p,\g}; 
\end{align*} 

 
For the terms involving  \(\Delta_{i,j,\f\theta {\sqrt{t}}}^r (f-g)\) and  \(\Delta_{i,j,\f\theta {\sqrt{t}}}^r g\), we make use of the identity:  
$$
\Big \|\Delta_{i,j,\frac{\theta}{\sqrt{t}}}^r g\Big \|_{p,\gamma}^p = \sb_\g \int_0^1 t^{d-1}
   \int_{\mathbb{S}_t^{d-1}} \big|\Delta_{i,j,\frac{\theta}{\sqrt{t}}}^r g(t\xi)\big|^p \, \d\sigma(\xi) \, \sw_{d-1,\gamma}(t) \, \d t.
$$
The first estimate for \(\Delta_{i,j,\frac{\theta}{\sqrt{t}}}^r (f - g)\) follows directly from \eqref{eqn-3.4}.  
To estimate the term \(\Delta_{i,j,\frac{\theta}{\sqrt{t}}}^r g\), it suffices to establish the following intermediate inequality.  
Let \(g_t(\xi) = g(t\xi, t)\) for \(0 < t < 1\). Then  
\[
\left\|\Delta_{i,j,\frac{\theta}{\sqrt{t}}}^r g_t(\xi)\right\|_{L^p(\mathbb{S}^{d-1})} \leq
c\, \theta^r \left\|\frac{1}{\sqrt{t}^r} D_{i,j}^r g_t(\xi)\right\|_{L^p(\mathbb{S}^{d-1})},
\]
which has been established in \cite[Lemma 2.6 (ii)]{DaiX2} upon replacing \(\theta\) with \(\frac{\theta}{\sqrt{t}}\).

The proof of inequalities about $\Delta_{\t\varphi}^r g$, $\Delta_{\t\varphi}^r (f-g)$
reduces to the corresponding inequalities in one variable, which are given in \cite[p. 58]{DT}. Indeed,
$\|\Delta_{\t\varphi}^r( f - g)\|_{L^p(\VV_{r\rho}, \sw)}$ and $\|\Delta_{\t\varphi}^r g\|_{L^p(\VV_{r\rho}, \sw)}$
correspond to the main part of the Ditzian–Totik modulus of smoothness. The first term can be treated similarly to the proof of part (i) in Lemma~\ref{prop-moduleA}. For the second term, by setting
$g_{\xi}(t) := g(t\xi, t)$, we obtain 
\begin{align*}
\|\Delta_{\t\varphi}^r g\|_{L^p(\VV_{rp},\sw)}^p&=\int_{\SS^{d-1}}\int_{J_{r\rho}} |\Delta_{\t\varphi}^r g_{\xi}(t)|^p \varpi_{d-2,\g}
(t)\d t \d\s(\xi)\\&\leq c\,\t ^{rp}\int_{\SS^{d-1}}\int_0^1 |\varphi^r(t)\partial^r_t g_{\xi}(t)|^p \varpi_{d-2,\g}
(t)\d t \d\s(\xi)  = c\, \t^{rp}\|\varphi^r\partial_t^rg\|^p_{p,\g},
 \end{align*} 
where the inequality follows from the main-part estimate on the interval $[0,1]$ stated in \cite[p. 58]{DT}. Thus, both inequalities
in \eqref{eq:o_r<K_r} hold, which completes the proof of the inequality 
$\o_r(f; \rho)_{p,\sw_{-1,\g}} \leq c \sK_r(f; \rho)_{p,\sw_{-1,\g}}$. 
	
Next we prove the reversed inequality $K_r(f; \rho)_{p,\sw_{-1,\g}}  \leq c \o_r(f; \rho)_{p,\sw_{-1,\g}}$ and we 
follow the approach of the proof in \cite[Theorem 6.10]{DaiX2}. Setting $n = \lfloor \f1\rho \rfloor$ and choose 
$g = \sL_n^\g f$, we obtain
\begin{align}\label{eqn2.21}
    \sK_r(f; \rho)_{p,\sw_{-1,\g}}  \leq \|f-\sL_n^\g f \|_{p,\g} \,& + \rho^r 
        \max_{1\leq i<j\leq d}\Big \|\f{1}{\sqrt{t}^r}D_{i,j}^r\sL_n^\g f \Big \|_{p,\g} \\
      & +\rho^r\|\varphi^r\partial_t^r \sL_n^\g f \|_{p,\g}. \notag
\end{align} 
The first term on the right-hand side is bounded by $c\,  \o_r(f; \rho)_{p,\sw_{-1,\g}} $ by the direct estimate in 
Theorem \ref{main-thmV0} (i). For the second term,  
\begin{align*}
    \rho^r \Big\|\f{1}{\sqrt{t}^r}D_{i,j}^r \sL_n^\g f \Big \|_{p,\g} 
   & = \left[  \int_0^1 t^{d-1} \rho^r \f{1}{\sqrt{t}^r} \int_{\sph}|D_{i,j}^r\sL_n^\g f(t\xi,t)|^p \d\sigma(\xi) \sw_{-1,\g}(t) \d t \right]^{1/p} \\ 
   & \le c \sup_{|\t|\leq \f1{n\sqrt{t}}} \left (\int_0^1 t^{d-1} \int_{\sph}|\Delta_{i,j,{\t}}^r \sL_n^\g f(t\xi,t)|^p \d\sigma(\xi)
      \sw_{-1,\g}(t) \d t\right)^{\f 1p}\\
      & \le c \left (\int_0^1 t^{d-1} n\sqrt{t}\int_{0}^{\f1{n\sqrt{t}}}\|\Delta_{i,j,{\t}}^r \sL_n^\g f(t\xi,t)\|_p^p \d\t
      \sw_{-1,\g}(t) \d t\right)^{\f 1p}\\
 & \le c \left (\int_0^1 t^{d-1} n \int_{0}^{\f1{n}}\Big \|\Delta_{i,j,{\f\t{\sqrt{t}}}}^r \sL_n^\g f(t\xi,t)\Big\|_p^p \d\t
 \sw_{-1,\g}(t) \d t\right)^{\f 1p}\\
   & \leq  c \o_r(\sL_n^\g f; n^{-1})_{p, \sw_{-1,\g}}
\end{align*}    
where the first inequality follows from     \cite[Lemma 2.6(iii)]{DaiX2},  the second inequality follows from \cite[Lemma 2.6(iv)]{DaiX2}, and the third inequality follows by replacing \(\theta\) with \(\frac{\theta}{\sqrt{t}}\).  
Thus, applying the triangle inequality and  Theorem \ref{main-thmV0} (i), we obtain
\begin{align*}
 \rho^r \Big \|\f{1}{\sqrt{t}^r}D_{i,j}^r \sL_n^\g f\Big \|_{p,\g} & \leq  c\, \o_r(\sL_n^\g f - f; n^{-1})_{p, \sw_{-1,\g}} 
 +   c\,\o_r( f; n^{-1})_{p, \sw_{-1,\g}}\\ 
   & \leq c\|f-\sL_n^\g f  \|_{p,\g}+c\,\o_r( f; n^{-1})_{p, \sw_{-1,\g}}\\&\leq c\, \o_r( f; \rho)_{p, \sw_{-1,\g}}. 
\end{align*}
For the third term on the right-hand side of \eqref{eqn2.21}, we apply the relation below that holds for $g \in\Pi_n^d$
\begin{align*}
   n^{-rp}\left \|\varphi^r \partial_t^r g \right \|_{p,\g}^p \sim\, & n  \int_0^{n^{-1}} \|\Delta_{\t\varphi}^r g \|^p_{L^p(\VV_{r\rho}, \sw)} d\t .
\end{align*} 
The relation follows from the corresponding result in one variable, where the equivalence follows from 
\cite[p. 57]{DT}. Applying the relation with $g = \sL^\g_n f$, we obtain
$$
 n^{-r} \left \|\varphi^r \partial_t^r \sL_n^\g f \right \|_{p, \sw_{-1,\g}}
       \le c \, \o_r \left(\sL_n^\g f; n^{-1}\right) \le c\, \o_r( f; n^{-1})_{p, \sw_{-1,\g}}.
$$
Putting these estimates together completes the proof. 

\subsection{Proof of inverse theorem} 
The proof of the inverse theorem is standard and follows from the Bernstein inequality. For the weight $L^p$ space 
with $\sw_{-1,\g}$ weight on the conic surface, we need the following new Bernstein inequalities, for $f \in \Pi_n(\VV_0^{d+1})$, and for $r=1,2,$ 
\begin{equation}
 \Big\|\f{1}{\sqrt{t}^r}D_{i,j}^r f\Big \|_{p,\gamma}\leq c n^r\ \| f \|_{p,\gamma} \quad\quad \text{and} \quad\quad
   \left \|\varphi^r\partial_t^r f \right \|_{p,\gamma}\leq c n^r\|f \|_{p,\gamma},
\end{equation}  
which are established in \cite[Theorem 2.1]{X23}. Using these inequalities, the standard proof gives the inverse 
estimate for the $K$-functional, 
$$
   \sK_r(f; \rho)_{p,\g } \le c n^{-r} \sum_{k=0}^n (k+1)^{r-1}\sE_k(f)_{p,\g},
$$ 
which gives (iii) in Theorem \ref{main-thmV0} by the equivalence in (ii) of the same theorem. 	

\section{Best approximation on the solid cone}
\setcounter{equation}{0}

In this section, we study approximation on the solid cone $\VV^{d+1}$. In the first subsection, we review
orthogonal polynomials and their associated kernels on the solid cone, and discuss their close relation 
to the corresponding quantities on the conic surface $\VV_0^{d+1}$. In the second subsection, we define 
a new pair of modulus of smoothness and $K$-functional for the solid cone, which are extrapolated from 
results on the conic surface in the previous section. The third section contains the main result on 
characterizing the best approximation over the solid cone, utilizing the newly introduced modulus of 
smoothness and $K$-functional. 

\subsection{Orthogonal polynomials and localized kernels}\label{sec:OPcone}
Let $\Pi_n^d$ denote the space of polynomials of degree at most $n$ in $d$ variables. For $\mu > -\f12$ 
and $\g > -1$, we define the weight function $W_{\g,\mu}$ by
\begin{equation}\label{weightcone}
	W_{\g,\mu}(x,t) = (t^2-\|x\|^2)^{\mu-\f12} (1-t)^\g, \qquad \|x\| \le t, \quad 0 \le t \le 1. 
\end{equation}
Orthogonal polynomials with respect to $W_{\g,\mu}$ on $\VV^{d+1}$ are studied in \cite{X20}, which are 
orthogonal with respect to the inner product 
$$
\la f, g\ra_{W_{\g,\mu}} =\sb_{\g,\mu} \int_{\VV^{d+1}} f(x,t) g(x,t) W_{\g,\mu} \d x \d t.
$$ 
Let $\CV_n(\VV^{d+1},W_{\g,\mu})$ be the space of orthogonal polynomials of degree $n$. Then 
$$
\dim \CV_n(\VV^{d+1},W_{\g,\mu})  = \binom{n+d+1}{n}, \quad n=0, 1,2,\ldots.
$$
An orthogonal basis of $\CV_n(\VV^{d+1}, W_{\g,\mu})$ can be given in terms of the Jacobi polynomials 
and the orthogonal polynomials on the unit ball. For $m =0,1,2,\ldots$, let $\{P_\kb^m(W_\mu): |\kb| = m, 
\kb \in \NN_0^d\}$ be an orthonormal basis of $\CV_n(\BB^d, W_\mu)$ on the unit ball. Let
\begin{equation} \label{eq:coneJ}
	\Jb_{m,\kb}^n(x,t):= P_{n-m}^{(2m+2\mu+d-1, \g)}(1- 2t) t^m P_\kb^m\left(W_\mu; \frac{x}{t}\right), \quad 
\end{equation}
Then $\{\Jb_{m,\kb}^n(x,t): |\kb| = m, \, 0 \le m\le n\}$ is an orthogonal basis of $\CV_n(\VV^{d+1},W_{\g,\mu})$.
These polynomials are the eigenfunctions of a second-order differential operator $\fD_{\g,\mu}$ with the eigenvalues 
depending only on the degree; see \cite[Theorem 3.2]{X20}.  
 
The reproducing kernel of the space $\CV_n(\VV^{d+1}, W_{\g,\mu})$, denoted by $\Pb_n(W_{\g,\mu};\cdot,\cdot)$, 
can be written in terms of the above basis, 
$$
\Pb_n\big(W_{\g,\mu}; (x,t),(y,s) \big) = \sum_{m=0}^n \sum_{|\kb|=n} 
\frac{  \Jb_{m, \kb}^n(x,t)  \Jb_{m, \kb}^n(y,s)}{\la  \Jb_{m, \kb}^n,  \Jb_{m, \kb}^n \ra_{W_{\g,\mu}}}.
$$
It is the kernel of the projection $\proj_n(W_{\g,\mu}): L^2(\VV^{d+1},W_{\g,\mu}) \to 
\CV_n(\VV^{d+1}, W_{\g,\mu})$, 
\begin{align}\label{sc-proj}
	\proj_n(W_{\g,\mu};f) = \int_{\VV^{d+1}} f(y,s) \Pb_n\big(W_{\g,\mu}; \,\cdot, (y,s) \big)  W_{\g,\mu}(s) \d y \d s.
\end{align}
This kernel enjoys an addition formula that can be written as a triple integral over $[-1,1]^3$ for $\mu > 0$; see 
\cite[Theorem 4.3]{X20}. Moreover, let $\wh a$ be an admissible cut-off function.  For $(x,t), \, (y,s) \in \VV^{d+1}$, 
the localized kernel $\Lb_n(W_{\g,0}; \cdot,\cdot)$ is defined by  
$$
\Lb_n^\g (\big (x,t),(y,s)\big) := \sum_{j=0}^\infty \wh a\left( \frac{j}{n} \right)
\Pb_j\left(W_{\g,0}; (x,t), (y,s)\right). 
$$
Like the kernel $\sL_n$ on the conic surface, the kernel $\Lb_n^\g$ is highly localized via the distance 
function $\sd_{\VV}(\cdot, \cdot)$ on the cone; see  \cite[Theorem 5.9]{X21}. 

With these properties, the general framework in \cite{X21} applies to the solid cone. In particular, we have 
an analog of the modulus of smoothness $\widehat{\o}_r(f; t)_{p,\sw_{-1,\g}}$, defined as a multiplier operator, 
and the $K$-functional $\wh K_r(f;t)$, defined via the spectral operator $\fD_{\g,\mu}$, on the solid cone, 
which are equivalent and can be used to establish a characterization of the best approximation by polynomials,
see \cite[Section 5.7]{X21}. 
 
Our goal in this section is to extend the new modulus of smoothness and $K$-functional on the conic surfaces 
to the solid cone, and use them to characterize the error of best approximation. Instead of replicating the 
approach of the previous section, it is much easier to make use of a connection between the $L^p(\VV^{d+1}, W_\g)$ 
on the solid cone and the space $L^p(\VV_0^{d+2}, \sw_{-1,\g})$, where we set
$$
     W_\g(x,t) = W_{\g,0} (x,t) = (t^2-\|x\|^2)^{-\f12} (1-t)^\g,
$$
so that $\mu = 0$ from now on.  
 
Let $(x,t)$ and $(y,s)$ be the elements of the solid cone $\VV^{d+1}$. Let $x_{d+1} = \sqrt{t^2-\|x\|^2}$ for $\|x\| \le t$ 
and $y_{d+1} = \sqrt{s^2-\|y\|^2}$   for $\|y\| \le s$. We shall adopt the notation 
\begin{equation}\label{eq:XY}  
X:=(x,x_{d+1}), \quad Y:=(y,y_{d+1}), \quad Y_*:=(y, - y_{d+1})
\end{equation} 
throughout this section. It is evident that $\|X\| = t$, $\|Y\| = \|Y_*\| =s$, so that $(X,t)$, $(Y,s)$ and $(Y_*,s)$ 
are elements of the conic surface $\VV_0^{d+2}$ in $\RR^{d+2}$. 

We start with a relationship between the kernels on \(\VV^{d+1}\) and \(\VV_0^{d+2}\) established in \cite[Proposition 4.3]{X23}.

\begin{lem}\label{lem3.4}
Let $\sL_n^\g(\cdot,\cdot)$ be the kernel defined in \eqref{def:Ln-kernel-V0} for $\VV_0^{d+2}$. Then
\begin{equation} \label{eq:LnW-Lnw}
   \Lb_n^\g\big((x,t), (y,s)\big) = \sL_n^\g\big((X,t), (Y,s)\big) + \sL_n^\g \big((X,t), (Y_*,s)\big). 
\end{equation}
\end{lem}

Given a function $f$ on $\VV^{d+1}$, we consider the following extension of $f$, 
\begin{equation}\label{eq:XY2}
	\tilde{f}\big(X, t\big)=\tilde{f}\big((x,x_{d+1}), t\big):=f(x,t), \quad  (x,t) \in \VV^{d+1}, \quad
	\big(X, t\big)\in\VV_0^{d+2},
\end{equation} 
so that $f$ can be regarded as a function on $\VV_0^{d+2}$.  The connection between the cone $\VV^{d+1}$
and the conic surface $\VV_0^{d+2}$ is further manifested in the following integral relation 
\cite[Proposition 4.4]{X23}. 

\begin{lem}\label{prop:IntV0V}
Let $f: \RR^{d+1} \mapsto \RR$ be a continuous function. Let $x_{d+1} = \sqrt{t^2-\|x\|^2}$ for $\|x\| \le t$.
Then, with $X= (x,x_{d+1})$ and $X_* = (x,-x_{d+1})$, 
$$
	\int_{\VV_0^{d+2}} f(y,s) \sw_{-1,\g}(s) \d \sm(y,s) = \int_{\VV^{d+1}} 
	\big [ f\big(X,t\big) +  f\big(X_* ,t\big)\big ] W_{\g}(x,t) \d x\d t.
$$
\end{lem}

For any $(x,t),(y,s)\in\VV^{d+1}$, we define  the   operator $\Lb_n^\g f $ by
\begin{equation}\label{nearbest3}
	\Lb_n^\g f (x,t) :=  \bb_\g \int_{\VV^{d+1}}f(y,s)\Lb_{n}^\g \big((y,s),(x,t)\big) W_{\g}(y,s)  \d \sm(y,s),
\end{equation} 
which is a polynomial of near-best approximation. It satisfies $\Lb_n^\g f  \in  \Pi_{2n}\big(\VV^{d+1}\big)$ and 
$\Lb_n^\g f  = f $ for $f \in \Pi_n\big(\VV ^{d+1}\big)$, and 
$$
 \|f- \Lb_n^\g f  \|_{p,W_{\g}} \le c \Eb_n(f)_{p,W_{\g}}, \qquad 1 \le p \le \infty,
$$
where $\Eb_n(f)_{p,W_{\g}}$ is  the error of best approximation defined by 
\begin{align}\label{conebestapprox}
	\Eb_n(f)_{p, W_\g}:= \inf_{g \in \Pi_n^{d+1}} \|f - g\|_{p, W_\g}, \qquad 1 \le p \le \infty.
\end{align}

The near-best operators on the cone and the conic surface are closely related. For \(X = (x, x_{d+1})\) and 
\((X, t) \in \VV_0^{d+2}\), let us denote by $\sL_n^\g \tilde f$ the operator defined in \eqref{def:Ln-surface} for 
\(\VV_0^{d+2}\). Specifically, it is given by
\begin{equation}\label{nearbest-conicsurface}
	\sL_n^\g \tilde{f}(X,t):=  \sb_\g \int_{\VV_0^{d+2}} \tilde{f}(Y,s) \sL_n^\g \big( (X,t), (Y,s) \big) \sw_{-1,\g}(s) \d \sm(Y,s).
\end{equation}

\begin{lem} \label{operator-connection} 
Let $\tilde f$ be defined as in \eqref{eq:XY2}. For $(x,t)\in\VV^{d+1}$, $( (x,x_{d+1}),t) \in\VV_0^{d+2}$, 
$$
     \sL^\g_n\tilde{f} \big( (x,x_{d+1}\big),t)= \Lb^\g_nf(x,t).
$$ 
\end{lem}

\begin{proof}  
By the Definitions of $\sL^\g_n\tilde{f}(X,t)$ and $\Lb^\g_nf(x,t)$, and using the Lemmas \ref{lem3.4} and \ref{prop:IntV0V}, 
we obtain 
\begin{align*}
& \sL^\g_n \tilde{f} \big( X,t) = \int_{\VV_0^{d+2}} \tilde{f}\big(Y, s\big) \sL_n^\g \big( (X,t), (Y,s)\big)\sw_{-1,\g}(s) \d\sigma(y,y_{d+1}) \d s   \\ & = \int_{\VV^{d+1}} \Big[ \tilde{f}(Z,u) \sL_{n}^\g\big((Z,u),(x,t)\big) + \tilde{f}\big(Z_*,u \big)
 \sL_{n}\big(\big(Z_*,u\big),(x,t)\big) \Big] W_{\g,0}(z,u)\d z \d u\\
  &=  \int_{\VV^{d+1}}f(z,u)\Lb_n^\g\big((x,t), (z,u)\big)W_{\g,0}{(z,u)}\d z \d u =  \Lb^\g_nf(x,t),
\end{align*}
where $Z=(z, z_{d+1})$ and $Z_* = (z,- z_{d+1})$.
\end{proof}

\subsection{A new modulus of smoothness and $K$-functional}
We denote by $\|\cdot\|_{p,W_\g}$ the weighted $L^p$ norm on $\VV^{d+1}$
$$
\|f\|_{p,W_\g} = \left( \int_{\VV^{d+1}} | f(x,t)|^p W_{\g}(x,t) \d \sm (x,t)\right)^{1/p}, \qquad 1 \le p < \infty, 
$$
where $\sm$ denotes the Lebesgue measure on the cone. For $ p =\infty$, we denote by $\|\cdot\|_\infty$ the 
uniform norm on $\VV^{d+1}$.  

\begin{defn}\label{def:modulusV}
Let $f\in L^p(\VV^{d+1}, W_{\g,\mu})$, $1\leq p <  \infty,$ and $f\in C(\VV^{d+1})$ if $p=\infty$. 
Setting $\tilde{f}$ be the extension of $f$ as in \eqref{eq:XY2}.  For $\varphi(t)=\sqrt{t(1-t)}$, $\rho \geq 0$, the modulus of smoothness on the cone is defined by 
\begin{align*}
\o_r(f; \rho)_{p,W_{\g}}   &:= \sup_{0 < \t \le \rho}   \max_{1\leq i < j \leq d }
	\Big \| \Delta_{i,j,\f{\t}{\sqrt{t}}}^r f\Big\|_{p, W_{\g}} \\
   & + \sup_{0 < \t \le \rho}   \max_{1\leq i  \leq d }\Big \| \Delta_{i,d+1,\f{\t}{\sqrt{t}}}^r \tilde{f}\Big\|_{L^p(\VV_0^{d+2},\sw_{-1,\g})  } 
      + \sup_{0<\t \le \rho} \Big\| \Delta_{\t\varphi}^r \tilde{f}\Big\|_{L^p(I_{r\rho},\sw_{-1,\g}) } ,
\end{align*} 
where the last  norm is taken over the interval indicated with $I_{t}:=[t^2,1-t^2]$ with the usual change when $p=\infty$. 
\end{defn} 

This modulus of smoothness is compatible with the modulus of smoothness on the conic surface $\VV_0^{d+2}$. 
To emphasize the dimension of the domain, we write
$$
\o_r(f; \rho)_{L^p(\VV_0^{d+2},  \sw_{-1,\g})} = \o_r(f; \rho)_{p,\sw_{-1,\g}}   \, \, \hbox{and} \, \, 
  \o_r(f; \rho)_{L^p(\VV_0^{d+2}, W_\g)} =\o_r(f; \rho)_{p,W_\g},
$$
in the statement of the following proposition. 

\begin{prop} \label{prop:moduli-equ}
Assume $f \in L^p(\VV^{d+1}, W_\g)$ if $1\le p < \infty$ and $f \in C(\VV^{d+1})$ if $p = \infty$. Let $\wt f$ be defined
as in \eqref{eq:XY2}. Then 
$$
 \o_r(f, \rho)_{L^p(\VV^{d+1}, W\g)} \sim   \o_r (\wt f; \rho)_{L^p(\VV_0^{d+2}, \sw_{-1,g})}.
$$
\end{prop}

\begin{proof}
Comparing Definitions \ref{def:modulusV0} and \ref{def:modulusV} shows that we only need to consider 
$\Big \|\Delta_{i,j,\f{\t}{\sqrt{t}}}^r \wt f\Big \|_{L^p(\VV_0^{d+2}, \sw_{-1,\g})}$ for $1\le i,j\le d$ in 
$\o_r (\wt f; \rho)_{L^p(\VV_0^{d+2}, \sw_{-1,\g})}$. For this range of $i$ and $j$, 
$\Delta_{i,j,\f{\t}{\sqrt{t}}}^r f(x,t) = \Delta_{i,j,\f{\t}{\sqrt{t}}}^r \wt f(X,t)$ by definition. 
For $(x,t) \in \VV^{d+1}$, we let $x = t x'$ with $x' \in \BB^d$. Then 
$$
   W_\g (t x', t) = (1-t)^\g \big(t^2 - \|t x' \|^2 \big)^{-\f12} = \frac{w_{-1,\g}(t)} {\sqrt{1-\|x'\|^2}}. 
$$ 
Hence, making this change of variables, we obtain
\begin{align*}
\Big \|\Delta_{i,j,\f{\t}{\sqrt{t}}}^r f\Big \|_{L^p(\VV^{d+1}, W_\g)} & = \sb_\g \int_{\VV^{d+1}}
    \Big | \Delta_{i,j,\f{\t}{\sqrt{t}}}^r f \Big|^p W_\g(x,t) \d x \d t\\
& = \sb_\g \int_0^1 t^d \int_{\BB^d} \Big | \Delta_{i,j,\f{\t}{\sqrt{t}}}^r f (t x',t) \Big|^p  \frac{w_{-1,\g}(t)} {\sqrt{1-\|x'\|^2}} \d x' \d t\\
& =  \sb_\g \int_0^1 t^d \int_{\SS^d} \Big | \Delta_{i,j,\f{\t}{\sqrt{t}}}^r \wt f (t X,t) \Big|^p  w_{-1,\g}(t)  \d\sm (X,t) \\
& =  \Big \|\Delta_{i,j,\f{\t}{\sqrt{t}}}^r \wt f\Big \|_{L^p(\VV_0^{d+2}, \sw_{-1,\g})}, 
\end{align*}
where the third identity follows from the integral relations between $\BB^d$ and $\SS^{d+1}$, see \cite[(A.5.4)]{DaiX}.  
This proves the equivalence of the two moduli of smoothness. 
\end{proof}

\begin{defn}\label{def:KfuncV} 
Let $f\in L^p(\VV^{d+1}, W_{\g,\mu})$, $1\leq p <  \infty,$ and $f\in C(\VV^{d+1})$
if $p=\infty$. Setting $\tilde{g}$ as the extension of $g$ defined in \eqref{eq:XY2}.
For   $r =1, 2$ and $\rho\geq 0$, the K-functional is defined by 
\begin{align*}
 & \Kb_r(f,\rho)_{p,W_{\g}} :=   \inf_{g \in C^r(\VV^{d+1})} 
	\bigg\{  \|f-g\|_{p,W_{\g}}  + \rho^r\left\|\varphi^{ r }\partial_t^rg  \right\|_{p,W_\g} \\ 
	& \qquad\qquad+ \rho^r \max_{1\leq i<j\leq d}   \Big \|\frac 1{\sqrt{t}^r}D_{i,j}^r g \Big \|_{p,W_{\g}}   
	+ \rho^r \max_{1\leq i \leq d} \Big \|\frac 1{\sqrt{t}^r} \big ( \Phi \partial_i \big)^r g \Big \|_{L^p(\VV^{d+1},W_\g)}  
\bigg\},
\end{align*}
where $\varphi(x,t) = \sqrt{t(1-t)}$ and $\Phi(x,t) = \sqrt{t^2 - \|x\|^2}$. 
\end{defn} 

This $K$-functional is closely related to the $K$-function defined in \eqref{eq:cs-K-func} on the conic surface 
$\VV_0^{d+2}$. As in the case of moduli of smoothness, we write
$$
\sK_r(f; \rho)_{L^p(\VV_0^{d+2},  \sw_{-1,\g})} = \sK_r(f; \rho)_{p,\sw_{-1,\g}}   \, \, \hbox{and} \, \, 
  \Kb_r(f; \rho)_{L^p(\VV^{d+1}, W_\g)} =\Kb_r(f; \rho)_{p,W_\g},
$$

\begin{prop}\label{prop:K-equ}
Assume $f \in L^p(\VV^{d+1}, W_\g)$ if $1\le p < \infty$ and $f \in C(\VV^{d+1})$ if $p = \infty$. Let $\wt f$ be defined
as in \eqref{eq:XY2}. Then 
$$
 \Kb_r(f, \rho)_{L^p(\VV^{d+1}, W\g)} \sim  \sK_r \big(\wt f; \rho\big)_{L^p(\VV_0^{d+2}, \sw_{-1,\g})}.
$$
\end{prop}

\begin{proof}
Using Lemma \ref{prop:IntV0V}, we see that $\|f-g\|_{L^p(\VV^{d+1},W_\g)} = \|\wt f - \wt g\|_{L^p(\VV_0^{d+1}, \sw_{-1,\g})}$.  
Hence, comparing Definitions \ref{def:KfuncV} and \ref{eq:cs-K-func}, we only need to consider the term 
$\Big \|\frac 1{\sqrt{t}^r} \big ( \Phi \partial_i \big)^r g \Big \|_{L^p(\VV^{d+1},W_\g)}$ for $1 \le i \le d$ 
in $\Kb_r(f,\rho)_{p,W_{\g}}$. Let $(x,t) \in \VV^{d+1}$. Setting $X=(x,x_{d+1})$. Then $(X,t) \in \VV_0^{d+2}$ 
if $x_{d+1} = \Phi(x,t)$. For $r \in \NN$ and $(X,t) \in \VV_0^{d+2}$, we first claim that 
\begin{equation} \label{eq:FrXt}
F_r(X,t):=  (- \Phi \partial_i)^r g (x, t) = D_{i, d+1}^r \wt g(X ,t). 
\end{equation}
Since $D_{i,d+1}$ is an angular derivative, we obtain 
$$
   D_{i, d+1} \wt g(X,t) = (x_i \partial_{d+1} - x_{d+1} \partial_i) g(x,t) =   - x_{d+1} (\partial_i g)(y,t) = - \Phi(x,t) \partial_i g(x,t),
$$
which verifies \eqref{eq:FrXt} for $r =1$. Assume the identity has been proved for some $r\in \NN$. Then, taking the
derivative for $x_i$ and using $x_{d+1} = \Phi(x,t)$, we obtain 
\begin{align*}
 F_{r+1}(X,t)\,& =  - \Phi(x,t) \frac{\partial}{\partial x_i} F_r(X,t) \\
      & = - \Phi(x,t) \partial_i F_r(X,t) - \Phi(x,t) \partial_{d+1} F_r(X,t) \frac{\partial}{\partial x_i} \Phi(x,t) \\
      & = - x_{d+1} \partial_i F_r(X,t) + x_i \partial_{d+1} F_r(X,t) \\
      &  = D_{i, d+1} F_r(X,t) = D_{i, d+1}^{r+1} \wt g(X,t)
\end{align*}
by induction hypothesis, which completes the proof of \eqref{eq:FrXt}. With this identity, it follows readily, by 
Lemma \ref{prop:IntV0V}, that 
$$
  \Big \|\frac 1{\sqrt{t}^r} \big ( \Phi \partial_i \big)^r g \Big \|_{L^p(\VV^{d+1},W_\g)}
        =  \Big \|\frac 1{\sqrt{t}^r} D_{i,d+1}^r \tilde{g}\Big \|_{L^p(\VV_0^{d+2},\sw_{-1,\g})}.
$$
This proves the equivalence of the two $K$-functionals. 
\end{proof}

\subsection{Main result}
Our main results are the characterization of the error of best approximation to $f$ from 
$\Pi_n^{d+1}$, $\Eb_n(f)_{p, W_\g}$ be defined in \eqref{conebestapprox}, by the modulus of 
smoothness $\o_r(f; h)_{p, W_\g}$, which is an analog of Theorem \ref{main-thmV0}.

\begin{thm}
Let $\g \ge 0$, $f \in L^p(\VV^{d+1}, W_\g)$ if $1 \le p < \infty$ and $f\in C(\VV^{d+1})$ if $p = \infty$. 
Then for   $n =1,2,\ldots$, there hold 
\begin{enumerate} [ \rm  (i)]
\item direct estimate: 
$$
    \Eb_n(f)_{p,W} \le c \, \o_r(f;n^{-1})_{p,W}.
$$
\item inverse estimate:  for $r=1, 2$
$$
	\o_r(f;n^{-1})_{p,W} \le c n^{-r} \sum_{k=0}^n (k+1)^{r-1}\Eb_k(f)_{p, W}.
$$
\item equivalence:  for $r=1,2$,
$$
c_1 \Kb_r(f; \t)_{p,W_\g } \le \o_r(f;\t)_{p,W_\g} \leq  \Kb_r(f;\t)_{p,W_\g}.
$$
\end{enumerate}
\end{thm}
 
Using the relations between the near-best operators in Lemma \eqref{operator-connection}, we obtain 
$$
   f(x,t) -  \Lb^\g_nf(x,t) = \wt f\big( X,t) - \sL^\g_n\tilde{f} (X,t),
$$
from which the proof of the direct estimate follows immediately from (i) of Theorem \ref{main-thmV0} and 
Proposition \ref{prop:moduli-equ}. The proof of (ii) follows from the standard procedure, as in the proof
of Theorem \ref{main-thmV0}, by using the corresponding Bernstein inequality in \cite{X23}. The 
equivalence in (iii) follows immediately from Propositions \ref{prop:moduli-equ} and \ref{prop:K-equ},  
and (iii) of Theorem \ref{main-thmV0}.

\end{document}